\documentclass[11pt,letterpaper]{amsart}
\usepackage[utf8]{inputenc}
\usepackage[T1]{fontenc}
\usepackage{amssymb}
\usepackage{graphicx}
\usepackage{url}
\usepackage{enumerate}
\usepackage{enumitem}
\usepackage{verbatim}
\usepackage{color}
\usepackage{tikz}
\usetikzlibrary{arrows.meta,bending}
\usetikzlibrary{decorations.pathmorphing}
\usetikzlibrary{decorations.markings}
\usepackage{array}
\usepackage{tikz-cd}
\usepackage{float}
\usepackage{amsmath}
\usepackage{amsthm}
\usepackage[colorlinks=true,linkcolor=purple,citecolor=violet]{hyperref}
\usepackage{mathtools}
\usepackage{bbm}
\usepackage[capitalise, noabbrev]{cleveref}
\usepackage{fullpage}
\usepackage{lmodern}
\usepackage{esint}
\usepackage[mathscr]{euscript}
\usepackage{booktabs}
\usepackage{makecell}
\usepackage{tabularx}

\newcolumntype{C}[1]{>{\centering\arraybackslash}m{#1}}
\newcolumntype{Y}{>{\raggedright\arraybackslash}X}

\usetikzlibrary{positioning,arrows.meta}

\DeclareMathOperator{\bS}{\mathbb{S}}
\DeclareMathOperator{\bR}{\mathbb{R}}

\DeclareMathOperator{\bH}{\mathbb{H}}
\DeclareMathOperator{\bRH}{\mathbb{RH}}
\DeclareMathOperator{\bCH}{\mathbb{CH}}

\DeclareMathOperator{\mpp}{\mathfrak{p}}

\DeclareMathOperator{\ttg}{\mathtt{g}}

\DeclareMathOperator{\Ric}{\mathrm{Ric}}
\DeclareMathOperator{\Rm}{\mathrm{Rm}}
\DeclareMathOperator{\dvol}{d\mathrm{vol}}
\DeclareMathOperator{\Nil}{\mathrm{Nil}}
\DeclareMathOperator{\NN}{\mathscr{N}}
\DeclareMathOperator{\divv}{\mathrm{div}}

\usepackage{lipsum}

\newcommand{\vol}{\mathrm{vol}}

\newtheorem{theorem}{Theorem}[section]
\newtheorem{lemma}[theorem]{Lemma}

\newtheorem{proposition}[theorem]{Proposition}

\theoremstyle{definition}
\newtheorem{definition}[theorem]{Definition}
\newtheorem{question}[theorem]{Question}

\begin{document}

	\title{Poincar\'e--Einstein 4-manifolds with cusps}
	\author{Mingyang Li and Hongyi Liu}
    \dedicatory{Dedicated to Claude LeBrun on the occasion of his 70th birthday}

	\newcommand{\Addresses}{{
			\bigskip
			\footnotesize

			\textsc{{Simons Center for Geometry and Physics, Stony Brook University, Stony Brook, NY 11794, USA}}\par\nopagebreak
			\textit{E-mail address}: \href{mailto:mingyang.li@scgp.stonybrook.edu}{\texttt{mingyang.li@scgp.stonybrook.edu}}

			\par\medskip

			\textsc{{Department of Mathematics, Princeton University, Princeton,
					NJ 08544, USA.}}\par\nopagebreak
			\textit{E-mail address}: \href{mailto:hongyil@princeton.edu}{\texttt{hongyil@princeton.edu}}
	}}

	\begin{abstract}
		In this paper, we construct Poincar\'e--Einstein 4-manifolds with various kinds of cusps. In particular, we construct:
		\begin{enumerate}
			\item Infinite families of Einstein metrics on $(0,\infty)\times \NN$, where $\NN\to T^2$ is a principal $\bS^1$-bundle over $T^2$, with one Poincar\'e--Einstein end and one end asymptotic to a real or complex hyperbolic cusp.

			\item Infinite families of Einstein metrics on $(0,\infty)\times P$, where $P\to \Sigma_{\ttg}$ is a principal $\bS^1$-bundle over a closed Riemann surface $\Sigma_{\ttg}$ of genus $\ttg\geq 2$, with one Poincar\'e--Einstein end and one end asymptotic to a bundle of two-dimensional hyperbolic cusps over hyperbolic $\Sigma_{\ttg}$.
		\end{enumerate}
		Universal covers of (1) and (2) provide new complete negative Einstein metrics on $\bR^4$. These Einstein metrics also exhibit interesting degeneration phenomena. With this construction, we give a negative answer to a question of Anderson concerning cusp formation for Poincar\'e--Einstein 4-manifolds.
	\end{abstract}

	\maketitle

	\section{Introduction}

	In this paper, we study complete $4$-dimensional negative Einstein manifolds $(M,h)$ with Poincar\'e--Einstein and cusp ends. We always normalize the metric so that $\Ric_h=-3h$. We first introduce the Poincar\'e--Einstein asymptotic geometry for negative Einstein metrics.
	\begin{definition}
		Let $E$ be an end of a negative Einstein $4$-manifold $(M,h)$. We say that $E$ is Poincar\'e--Einstein (PE) if there are a connected closed $3$-manifold $N$, a diffeomorphism from $E$ to $(0,\epsilon)\times N$, and a defining function $\rho>0$ on $E$ such that the metric
		$g:=\rho^2h$ extends smoothly to $[0,\epsilon)\times N$, with $\rho=0$ and $d\rho\neq 0$ along $\{0\}\times N$.
	\end{definition}
	Although the conformal factor is not unique, the conformal class $[g|_{\{0\}\times N}]$ is well-defined, and this is the \emph{conformal infinity} of the Poincar\'e--Einstein end. Any such conformal metric $g$ is called a \emph{conformal compactification} for the Poincar\'e--Einstein end. If the Einstein manifold has only Poincar\'e--Einstein ends, we shall simply say the Einstein manifold is Poincar\'e--Einstein. Poincar\'e--Einstein manifolds are asymptotically hyperbolic in the sense that the sectional curvatures approach $-1$ at infinity.

	In the literature, the above Poincar\'e--Einstein asymptotic is also known as conformally compact Einstein (CCE) or asymptotically hyperbolic Einstein (AHE). The study of local existence, perturbative existence, uniqueness, compactness, and the construction of examples of Poincar\'e--Einstein metrics has undergone extensive development \cite{lebrun1980spaces,FeffermanGraham1985ConformalInvariants,FeffermanGraham2011AmbientMetric,LeBrun1982HSpace,LeBrun1991CompleteQuaternionicKahler,Biquard2002MetriquesAutoduales,Biquard2008ContinuationUnique,Lee2006FredholmOperators,chang2025problem,ChangGe2018CompactnessDimension4,ChangGeQing2020CompactnessII,ChangGeJinQing2024PerturbationCompactness,ChangYangZhang2025CylindricalInfinity,CalderbankSinger2004ComplexSingularities,GurskySzekelyhidi2020LocalExistence,Page1978CompactRotating,PagePope1987ComplexLineBundles,Pedersen1986SpinningTop,hitchin1995twistor}. Our recent work gives a non-perturbative construction of Poincar\'e--Einstein metrics that are conformally K\"ahler over complex line bundles over Riemann surfaces \cite{LiLiu2025PoincareEinstein}. For nondegenerate PE $4$-manifolds, in the sense that the linearized Einstein operator in Bianchi gauge admits no non-trivial $L^2$ kernel, there is a well-developed deformation theory \cite{Biquard2000EinsteinAsymptotiquementSymetriques,Lee2006FredholmOperators,BiquardRollin2009Wormholes}.

	On the other hand, negative Einstein $4$-manifolds may exhibit many other types of asymptotic geometry, and at present very little is known about Einstein metrics with such asymptotics. These metrics are of particular interest because they may arise as limits of negative Einstein manifolds under degeneration. The main focus of this paper is on the following asymptotic models.

	\begin{definition}\label{def:cusps}
		We say that an end $E$ of a complete negative Einstein $4$-manifold $(M,h)$ is
		\begin{itemize}
			\item asymptotically hyperbolic cuspidal,
			\item asymptotically complex hyperbolic cuspidal, or
			\item asymptotically $\Sigma_{\ttg}$-cuspidal, where $\Sigma_{\ttg}$ is a closed Riemann surface of genus $\ttg\geq 2$,
		\end{itemize}
		if there is a closed $3$-manifold $\mathscr L$, a suitable diffeomorphism from $E$ to $(R_0,\infty)\times \mathscr L$, and a model metric $h_0$ on $(R_0,\infty)\times \mathscr L$ such that, for some $\delta_0>0$ and every $k\geq 0$,
		\[
		|\nabla_{h_0}^k(h-h_0)|_{h_0}=O(e^{-\delta_0 r}).
		\]
		Here:
		\begin{itemize}
			\item In the asymptotically hyperbolic cuspidal case, $\mathscr L=T^3$, and
			\begin{equation}\label{eq:hyperbolic-cusp-model}
				h_0=dr^2+e^{-2r}g_{T^3},
			\end{equation}
			where $g_{T^3}$ is a flat metric on $T^3$.

			\item In the asymptotically complex hyperbolic cuspidal case, $\mathscr L=\Gamma\backslash \mathcal H^3$, where $\mathcal H^3$ is the $3$-dimensional Heisenberg group and $\Gamma\subset \mathcal H^3$ is a cocompact lattice. The model metric $h_0$ is the quotient metric induced by
			\begin{equation}\label{eq:complex-hyperbolic-cusp-model}
				\frac{1}{2}\left(dr^2+e^{-r}(dx^2+dy^2)+e^{-2r}(dt-x\,dy)^2\right)
			\end{equation}
			on $(R_0,\infty)\times \mathcal H^3$. Equivalently, $\mathscr L=\NN_\ell$ is the principal $\bS^1$-bundle over $T^2$ of degree $\ell\neq 0$ determined by $\Gamma$.

			\item In the asymptotically $\Sigma_{\ttg}$-cuspidal case,
			$
			\mathscr L=P_\ell,
			$
			where $P_\ell$ is the principal $\bS^1$-bundle over $\Sigma_{\ttg}$ of degree $\ell\in \mathbb Z$, and
			\begin{itemize}
				\item If $\ell=0$, then $P_0=\Sigma_{\ttg}\times \bS^1$, and
				\begin{equation}\label{eq:sigma-cusp-model-trivial}
					h_0=\frac{1}{3}\left(dr^2+e^{-2r}dt^2+g_{\Sigma_{\ttg}}\right),
				\end{equation}
				where $g_{\Sigma_{\ttg}}$ is a hyperbolic metric on $\Sigma_{\ttg}$.

				\item If $\ell\neq 0$, then $P_\ell=\Gamma\backslash \widetilde{SL}(2,\mathbb R)$ is a nontrivial principal $\bS^1$-bundle over $\Sigma_{\ttg}$, where $\Gamma\subset \widetilde{SL}(2,\mathbb R)$ is a cocompact lattice. The model metric $h_0$ is the quotient of
				\begin{equation}\label{eq:sigma-cusp-model-twisted}
					\frac{1}{3}\left(dr^2+e^{-2r}\left(dt+\frac{dx}{y}\right)^2+\frac{1}{y^2}(dx^2+dy^2)\right)
				\end{equation}
				on $(R_0,\infty)\times \widetilde{SL}(2,\mathbb R)$.
			\end{itemize}
		\end{itemize}
	\end{definition}

	Here $\mathcal H^3$ denotes the $3$-dimensional Heisenberg group
	\[
	\mathcal H^3=
	\left\{
	\begin{pmatrix}
		1 & x & t\\
		0 & 1 & y\\
		0 & 0 & 1
	\end{pmatrix}
	: x,y,t\in\mathbb R
	\right\},
	\]
	and $\widetilde{SL}(2,\mathbb R)$ is the universal cover of $SL(2,\mathbb R)$. The model metrics \eqref{eq:hyperbolic-cusp-model}--\eqref{eq:sigma-cusp-model-trivial} are Einstein, whereas \eqref{eq:sigma-cusp-model-twisted} is only asymptotically Einstein, with
	$
	\Ric(h_0)+3h_0=O(e^{-2r}).
	$
	The $r$-level cross-sections of \eqref{eq:hyperbolic-cusp-model}--\eqref{eq:sigma-cusp-model-twisted} carry, respectively, $\bR^3$, $\Nil$, $\bH^2\times \bR$, and $\widetilde{SL}(2,\mathbb R)$ geometry.

	For brevity, we shall refer to these ends simply as AH cusps, ACH cusps, and $\Sigma_{\ttg}$-cusps, respectively. We shall set $\NN_\ell$ as the nilmanifold that is the $\bS^1$-bundle over $T^2$ of degree $\ell$. Note that the sign of $\ell$ depends on the choice of orientation. In particular, $\NN_0=T^3$. The following three results realize each of these cusp geometries as one end of a complete negative Einstein $4$-manifold whose other end is Poincar\'e--Einstein.

\begin{theorem}\label{thm:PE-with-AH-cusp}
For every Riemannian metric $g^\sharp$ on $T^2$, there exists an
$\bS^1$-invariant Riemannian metric $g^\flat$ on $T^3$ whose quotient is
$(T^2,g^\sharp)$, such that $M=(0,\infty)\times T^3$ admits a complete
negative anti-self-dual Einstein metric $h$ with two ends $E_0,E_\infty$,
where $E_0$ is PE with conformal infinity $[g^\flat]$ and $E_\infty$ is an AH
cusp.
\end{theorem}

\begin{theorem}\label{thm:PE-with-ACH-cusp}
For every Riemannian metric $g^\sharp$ on $T^2$ and every
$\ell\in\mathbb{Z}_{>0}$, there exist $\bS^1$-invariant Riemannian metrics
$g^\flat_\nu$, $\nu=1,2$, on $\NN_\ell$ whose quotient is
$(T^2,g^\sharp)$, such that $M=(0,\infty)\times \NN_\ell$ admits complete
negative Einstein metrics $h_\nu$ with two ends $E_0,E_\infty$, where
\begin{itemize}
    \item $h_1$ is anti-self-dual, while $h_2$ is not anti-self-dual;
    \item $E_0$ is PE with conformal infinity $[g_\nu^\flat]$, and $E_\infty$
    is an ACH cusp.
\end{itemize}
\end{theorem}

\begin{theorem}\label{thm:PE-with-Sigma-cusp}
For every Riemannian metric $g^\sharp$ on a closed surface $\Sigma_{\ttg}$ of
genus $\ttg\geq 2$ and every $\ell\in \mathbb{Z}$, there exists an
$\bS^1$-invariant Riemannian metric $g^\flat$ on $P_\ell$ whose quotient is
$(\Sigma_{\ttg},g^\sharp)$, such that $M=(0,\infty)\times P_\ell$ admits a
complete negative Einstein metric $h$ with two ends $E_0,E_\infty$, where
$E_0$ is PE with conformal infinity $[g^\flat]$ and $E_\infty$ is a
$\Sigma_{\ttg}$-cusp.
\end{theorem}

	Here, on an oriented 4-manifold, an Einstein metric is called \emph{anti-self-dual} (ASD) if its self-dual Weyl curvature satisfies $W^+=0$. There is a previous construction of PE manifolds with intermediate-rank hyperbolic cusps based on perturbing hyperbolic examples \cite{BahuaudRochon2019GeometricallyFinite}. The Einstein metrics we constructed here are actually \emph{conformally K\"ahler} with an isometric $\bS^1$-action. There is a K\"ahler metric $g$ on $M$ that is conformal to $h$ and also invariant under the same $\bS^1$-action. The K\"ahler metric $g$ extends to $\overline M=[0,\infty)\times\mathscr{L}$, which in particular is a conformal compactification for the end $E_0$. The $g^\flat$ that appears in the above Theorems \ref{thm:PE-with-AH-cusp}--\ref{thm:PE-with-Sigma-cusp} is precisely the restriction of $g$ to $\{0\}\times\mathscr{L}$, while $(\Sigma,g^\sharp)$ is the quotient of $(\mathscr{L},g^\flat)$ by the isometric $\bS^1$-action.
    
     It is worth noting that, in Theorems \ref{thm:PE-with-AH-cusp}–\ref{thm:PE-with-Sigma-cusp}, the conformal infinity is of negative Yamabe type, except in the hyperbolic cusp case of Theorem \ref{thm:PE-with-AH-cusp}, where $(T^2,g^\sharp)$ and $(\NN_0,g^\flat)$ are flat. This follows from the results of Witten–Yau \cite{WittenYau1999} and Cai–Galloway \cite{Cai1999BoundariesOZ}, or from a direct argument on the conformal infinity.

     \ 
  
	Next we discuss applications of our construction. Given an Einstein manifold with an AH cusp, there is a well-known generalized Dehn filling procedure that was first proposed by Anderson \cite{Anderson2006DehnFillingHigherDim}, later refined by Biquard and Bamler \cite{Biquard2008DGGA,Bamler2012GeneralizedDehnFilling}, which allows us to close up the AH cusp by a so-called toral black hole metric. Note that the black hole metric is also conformally K\"ahler, and it turns out to be included in our previous construction of conformally K\"ahler PE manifolds \cite{LiLiu2025PoincareEinstein}. Previously, generalized Dehn filling has been performed on hyperbolic cusps, while in our situation AH cusps are only asymptotic to hyperbolic cusps. Note that there is also a continuous complex hyperbolic cusp closing procedure \cite{FuHeinJiang2025CuspClosing} based on complex geometry, where a complex hyperbolic cusp is closed by gluing a Tian-Yau gravitational instanton and a neck region constructed via the Calabi ansatz. Performing generalized Dehn filling on our PE manifolds with an AH cusp from Theorem \ref{thm:PE-with-AH-cusp}, we obtain
	\begin{theorem}\label{thm:Dehn-filling}
		For each PE manifold with an AH cusp $(M,h)$ that arises from Theorem \ref{thm:PE-with-AH-cusp}, there is a sequence of PE manifolds $(M_j,h_j)$ with the same conformal infinity of the PE end of $(M,h)$, such that after taking a pointed Gromov--Hausdorff limit at suitable base points of $(M_j,h_j)$, we recover $(M,h)$.
	\end{theorem}

	The following picture roughly illustrates the convergence of $(M_j,h_j)$ to $(M,h)$. There is a core torus $T^2$ in the black hole metric, so after filling approximately there is also a core torus in $M_j$. Taking the pointed Gromov--Hausdorff limit of $(M_j,h_j,p_j)$ we obtain the PE manifold with an AH cusp $(M,h,p)$. If we instead take pointed limits near the core torus, then the sequence of Einstein metrics is collapsing with bounded curvature, while locally after passing to the universal cover and taking a pointed limit we get the universal cover of the black hole metric.
	\begin{figure}[ht]
		\centering
		\begin{tikzpicture}[scale=0.7]
			\draw (0,0) ellipse (0.8 and 2);
			\draw (0,2).. controls (1,2) and (3,0.3).. (6,0.3);
			\draw (0,-2).. controls (1,-2) and (3,-0.3).. (6,-0.3);
			\draw (6,0.3) -- (8,0.3);
			\draw (6,-0.3) -- (8,-0.3);

			\draw (8,0.3) arc[start angle=90, end angle=-90, radius=0.3];

			\draw[red] (7.9,0.18).. controls (8.05,0)..(8.1,-0.18);

			\draw[->] (9,0) -- (10.7,0) node[midway,above] {$j\to\infty$};

			\draw (12,0) ellipse (0.8 and 2);
			\draw (12,2).. controls (13,2) and (14,0.1).. (19,0.05);
			\draw (12,-2).. controls (13,-2) and (14,-0.1).. (19,-0.05);

			\draw[->] (9,0) -- (10.7,0) node[midway,above] {$j\to\infty$};

			\node at (0,-3) {$(T^3,[g^\flat])$};

			\draw[->] (7.5,-2) to[out=60,in=-60] (8,-0.5);
			\node at (7.5,-2.5) {core torus};

			\draw[->] (4,-2) to[out=60,in=-60] (3,0);
			\node at (4,-2.5) {$p_j$};
			\draw[mark=*, mark size=0.8pt] plot coordinates {(3,0.1)};

			\draw[->] (16,-2) to[out=60,in=-60] (15,0);
			\node at (16,-2.5) {$p$};
			\draw[mark=*, mark size=0.8pt] plot coordinates {(15,0.1)};

			\node at (12,-3) {$(T^3,[g^\flat])$};

			\node at (19,0.5) {AH cusp};
		\end{tikzpicture}
		\caption{Cusp formation with fixed conformal infinity.}
		\label{fig:cusp-formation}
	\end{figure}
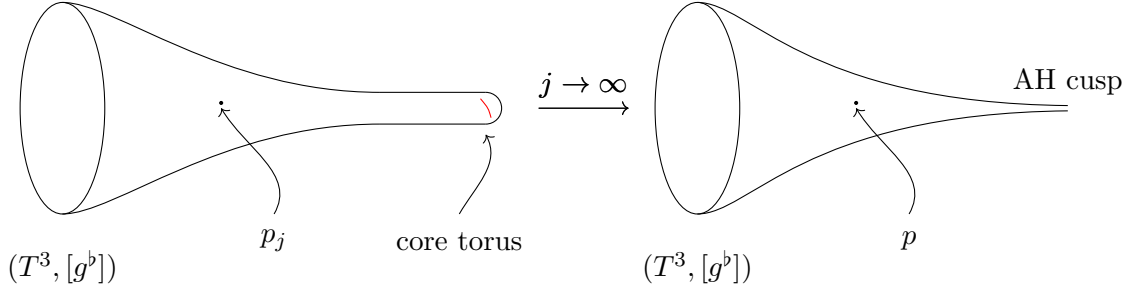

		The nontriviality of the above result lies in that the limit PE manifold with an AH cusp is \emph{non-hyperbolic} and during the degeneration, the conformal infinity is \emph{fixed} and \emph{not locally conformally flat} (Lemma \ref{lem:lemma-conformal-infinity-non-locally-conformally-flat}). Note that previously, generalized Dehn filling produced such degeneration processes but with a hyperbolic metric as the limit. The following question was asked and studied by Anderson \cite{anderson2003boundary}. Our Theorem \ref{thm:Dehn-filling} provides a negative answer.	
    \begin{question}[Anderson]
			Let $(M_j,h_j)$ be a sequence of PE 4-manifolds which develops an AH cusp in the limit. Must the limiting PE manifold with an AH cusp be hyperbolic?\footnote{We take this opportunity to clarify a point in the literature. Anderson \cite{anderson2003boundary} asserted that, under certain mild hypotheses, a limiting PE manifold with an AH cusp must be hyperbolic. However, there appears to be a gap in the proof. The argument involves translating into the cusp region and blowing up the difference between the Einstein metric and the hyperbolic cusp metric. It is claimed that the corresponding blow-up limit exists, but this need not be the case. For instance, examples constructed in the present paper converge to the hyperbolic cusp at a super-exponential rate, in which case the proposed blow-up limit does not arise.

            Craig \cite{Craig2006DehnFillingAHE} constructed counterexamples by applying generalized Dehn filling to hyperbolic manifolds with PE ends and AH cusp ends. Nevertheless, that paper still refers to the result of \cite{anderson2003boundary}, that such a limiting PE manifold with an AH cusp is hyperbolic. The examples in \cite{Craig2006DehnFillingAHE} are perturbative in nature and hence remain close to hyperbolic metrics, whereas the examples in the present paper may be far from hyperbolic, since the boundary data can be prescribed.}
	\end{question}

    \ 

	Now we briefly explain our approach. Suppose the negative Einstein manifold $(M,h)$ is \emph{non-trivially conformally K\"ahler}, in the sense that there is a complex structure on $M$ and a K\"ahler metric $g=e^{2f}h$ that is conformal to $h$ with the conformal factor being non-constant. Then associated to the conformally K\"ahler structure, there is always a Killing field $\mathcal{K}$ that preserves both $h$ and $g$, as observed by \cite{Przanowski1991KillingVectorFields,Derdzinski1983SelfDualKahler,Tod1997SUToda}. Assume that $\mathcal{K}$ induces an isometric free $\bS^1$-action on $M$, then by an ansatz essentially due to LeBrun and Tod \cite{LeBrun1991ExplicitSelfDual,Tod2006NoteASD,Tod1997SUToda}, the Einstein equation reduces to the following (twisted-)Toda equation
	\begin{equation}\label{eq:twisted-toda-with-k}
		(e^{v})_{\xi\xi}+\Delta_{\Sigma}{v}-2K_{\Sigma}=-\xi e^{v}\frac{12-6\xi \partial_{\xi}{v}}{12k^3+\xi^3},
	\end{equation}
	\begin{equation}\label{eq:W-with-k}
		W=\frac{12-6\xi \partial_{\xi}{v}}{12+\xi^3/k^3},
	\end{equation}
	where $k\in\{1,\infty\}$, $\Sigma$ is the K\"ahler reduction, and $\Delta_{\Sigma}$ is the Laplace operator of the constant curvature metric $g_\Sigma$ on $\Sigma$. The constant $K_\Sigma\in\{-1,0,1\}$ is the curvature of $g_\Sigma$. Here, the constant $k$ has geometric meaning that we shall see later, and the Einstein metric $h$ is anti-self-dual precisely when $k=\infty$. Once we have a solution to \eqref{eq:twisted-toda-with-k}, there is a conformally K\"ahler negative Einstein metric whenever $W>0$
	\begin{equation}\label{eq:h-Einstein-introduction}
		h=\xi^{-2}(Wd\xi^2+W^{-1}\eta^2+We^vg_{\Sigma})
	\end{equation}
	with $\eta$ being a connection 1-form for the $\bS^1$-bundle, which is determined by $v$. This approach eventually reduces the construction to boundary value problems for a (twisted-)Toda equation on suitable domains:
	\begin{equation}\label{eq:boundary-value-problem-introduction}
		\left\{\begin{aligned}
			&(e^{v})_{\xi\xi}+\Delta_{\Sigma}{v}-2K_\Sigma=-\xi e^{v}\frac{12-6\xi \partial_{\xi}{v}}{12k^3+\xi^3}, \\
			&v|_{\xi=0}=\varphi.
		\end{aligned}\right.
	\end{equation}
	Solving this problem on different domains, with different prescribed asymptotic behaviors of the solutions, produces Einstein metrics with corresponding asymptotic behaviors.
	Here, the boundary value $\varphi$ corresponds to the metric $g^\sharp$ that appears in Theorems \ref{thm:PE-with-ACH-cusp}--\ref{thm:PE-with-Sigma-cusp} via $g^\sharp=e^{\varphi}g_{\Sigma}$. For more details on this reduction, see Section \ref{sec:background}. This conformally K\"ahler perspective has been studied by \cite{LeBrun1996EinsteinMetricsComplexSurfaces,LeBrun2016EinsteinMaxwellConformallyKahler,LeBrun2012EinsteinHermitian4Manifolds}, and more recently by \cite{LeBrun2020BachFlatKahler,BiquardGauduchonLeBrun2024GravitationalInstantons,BiquardGauduchon2023ToricHermitianALF}

	Finally, we outline the structure of the paper. In Section \ref{sec:background}, we introduce the reduction of the Einstein equations. In Section \ref{sec:regularity}, we study the asymptotics of solutions to the (twisted-)Toda equation on cusp ends and formulate the corresponding boundary value problems. In Section \ref{sec:model-geometry}, we consider the simplest solutions to the (twisted-)Toda equation. We then solve these boundary value problems in Section \ref{sec:solve-boundary-value-problem}. Finally, in Section \ref{sec:degeneration}, we investigate various degenerations of our Einstein metrics.

	\subsection*{Acknowledgment}
	The first author would like to thank Claude LeBrun for helpful discussions. The second
  author thanks Sun-Yung Alice Chang and Paul Yang for many discussions on Poincar\'e--Einstein metrics. The authors thank Ruobing Zhang for the invitation to a workshop at the Mathematisches Forschungsinstitut Oberwolfach, where this work was initiated. The authors thank Tristan Ozuch for raising the question of whether the cusp asymptotics impose restrictions on the conformal infinity.
  
  We are both grateful to Claude LeBrun for his work on Einstein 4-manifolds, especially on the interplay between Einstein metrics, complex surfaces, and conformally K\"ahler geometry, which has been an important source of motivation for this work.

	\subsection*{Conventions}
	\begin{itemize}
		\item Given a compact Riemann surface $\Sigma$, we set $g_{\Sigma}$ as the metric with constant curvature $K_\Sigma\in\{-1,0,1\}$ on $\Sigma$. When $\Sigma=T^2$, we also normalize $g_{T^2}$ to have volume one.
		\item We denote a compact Riemann surface of genus $\ttg\geq2$ as $\Sigma_{\ttg}$.
		\item We write the real and complex hyperbolic spaces as $(\bRH^4,h_{\bRH^4})$ and $(\bCH^2,h_{\bCH^2})$, normalized to have Einstein constant $-3$.
		\item We denote the complex hyperbolic orientation of $\bCH^2$ as the positive orientation of $\bCH^2$, and the opposite orientation as the negative orientation.
		\item We denote by $\NN_\ell$ the $\bS^1$-bundle with degree $\ell$ over $T^2$.
	\end{itemize}

	\section{Background}
	\label{sec:background}
	In this section, we introduce the reduction of conformally K\"ahler Einstein metrics to the twisted-Toda equation. For more details on the reduction, see \cite{LiLiu2025PoincareEinstein}.

	Let $(M,h)$ be an oriented 4-manifold. With our choice of the orientation, one can decompose the bundle of 2-forms $\Lambda^2$ into the direct sum $\Lambda^2=\Lambda^+\oplus\Lambda^-$ of the $(\pm1)$-eigenspaces of the Hodge star operator. The sections of $\Lambda^+$ are called self-dual 2-forms and the sections of $\Lambda^-$ are called anti-self-dual 2-forms. The curvature tensor $\Rm_h$ then can be identified as a self-adjoint linear map $\Rm_h:\Lambda^+\oplus\Lambda^-\to\Lambda^+\oplus\Lambda^-$. It can also be decomposed into irreducible pieces
	\[\left(\begin{matrix}
		W_h^++\frac{s_h}{12}I & \Ric_h^0\\
		\Ric_h^0 & W_h^-+\frac{s_h}{12}I
	\end{matrix}\right)\]
	where $s_h$ is the scalar curvature, $\Ric_h^0=\Ric_h-\frac{s_h}{4}h$ is the trace-free part of the Ricci curvature, and $W_h^\pm$ are respectively the self-dual Weyl curvature and anti-self-dual Weyl curvature.

	Now we suppose $(M,h)$ is a 4-dimensional Einstein manifold that is non-trivially conformal to a K\"ahler metric $g:=\xi^2h$, equipped with the orientation induced by the complex structure. Denote the dual vector fields of a 1-form $\alpha$ with respect to $h$ and $g$ by $\alpha^{\sharp_h}$ and $\alpha^{\sharp_g}$, respectively. Then by \cite{Derdzinski1983SelfDualKahler,Tod2006NoteASD}, $(M,h)$ has to be in one of the following two situations.
	\begin{itemize}
		\item \textbf{Type I}: $W_h^+\equiv0$. That is, the Einstein metric $h$ is \emph{anti-self-dual} (ASD). In this situation the K\"ahler metric $g$ is \emph{K\"ahler scalar-flat}.
		\item \textbf{Type II}: $W_h^+$ is nowhere vanishing and has a repeated eigenvalue. In this situation, the K\"ahler metric $g$ is \emph{K\"ahler extremal} in the sense of Calabi, which means the vector field $(\bar\partial s_{g})^{\sharp_g}$ is a holomorphic vector field.
	\end{itemize}
	In both situations, a non-trivial Killing field $\mathcal{K}$ that preserves $h$ and $g$ naturally arises. The Killing field $\mathcal{K}$ moreover preserves the complex structure. We explain this separately.
	\begin{itemize}
		\item In the Type I case, set the K\"ahler form of $g$ as $\omega$, then the Killing field $\mathcal{K}$ is $(\delta_h(\xi^{-3}\omega))^{\sharp_h}$, where $\delta_h$ denotes the divergence of symmetric 2-forms under $h$ \cite{CalderbankPedersen2002TorusSymmetry}.
		\item In the Type II case, the Killing field $\mathcal{K}=J_g\nabla^g s_g$, equivalently the imaginary part of $-2(\bar\partial s_g)^{\sharp_g}$.
	\end{itemize}
	The conformal factor $\xi$ can also be determined explicitly. In the Type I case, the conformal factor $\xi$ up to scaling is $|(d (\mathcal{K}^{\sharp_h}))^+|_h^{-1}$, where $\mathcal{K}^{\sharp_h}$ denotes the 1-form that is dual to $\mathcal{K}$ under $h$. In the Type II case, the conformal factor $\xi$ up to scaling is the scalar curvature $s_g$. Since $\mathcal{K}$ is real holomorphic and Killing under $g$, it is also Hamiltonian; that is, there is a moment map $\sigma:M\to\bR$ such that $-d\sigma=\omega(\mathcal{K},\cdot)$. It turns out that up to scaling the conformal factor $\xi$ is also a moment map for $\mathcal{K}$. Hence by scaling $g$, we may suppose the conformal factor $\xi$ is precisely the moment map, and in the Type II case $\xi=s_g$, hence $g=\xi^2h=s_g^2h$. In the Type II case we moreover have $|s_g|=(2\sqrt{6}|W^+_h|_h)^{1/3}>0$.

	For the K\"ahler metric $(M,g)$ with the Hamiltonian Killing field $\mathcal{K}$, away from zeros of $\mathcal{K}$, one can perform K\"ahler reduction. If we further suppose the Killing field $\mathcal{K}$ induces a free $\bS^1$-action with the K\"ahler reduction being holomorphically a compact Riemann surface $\Sigma$, then the K\"ahler metric $g$ and the Einstein metric $h$ can be written as
	\begin{align}
		g&=Wd\xi^2+W^{-1}\eta^2+We^vg_{\Sigma},\label{eq:Kahler-metric-g}\\
		h&=\xi^{-2}(Wd\xi^2+W^{-1}\eta^2+We^vg_{\Sigma}),\label{eq:Einstein-metric-h}
	\end{align}
	where $g_{\Sigma}$ is the constant curvature metric with $K_\Sigma\in\{-1,0,1\}$ (with volume one if $\Sigma=T^2$), and $\eta$ is the connection 1-form determined by the metric $g$ on the $\bS^1$-bundle associated to the free $\bS^1$-action. Diffeomorphically, the manifold $M$ is $I\times\mathscr{L}$, where $I$ is an open interval in the range of $\xi$ and the link $\mathscr{L}$ is an $\bS^1$-bundle over $\Sigma$. Following the computation of LeBrun \cite{LeBrun1991ExplicitSelfDual}, we have that in the Type I and Type II cases, $h$ being Einstein is equivalent to $v,W$ satisfying the following \emph{(twisted-)Toda equation systems} respectively.

	\noindent\emph{Type I.}
	\begin{equation}\label{eq:case-I}
		\Delta_{\Sigma} v+(e^{v})_{\xi\xi}-2K_\Sigma=0,
		\qquad
		W=1-\frac12\xi v_\xi.
	\end{equation}
	\noindent\emph{Type II.}
	\begin{equation}\label{eq:case-II}
		\Delta_{\Sigma} v+(e^{v})_{\xi\xi}-2K_{\Sigma}=-\xi\,W e^{v},
		\qquad
		W=\frac{12-6\xi v_\xi}{12+\xi^{3}}.
	\end{equation}
	Here, the (twisted-)Toda equations follow from LeBrun's calculation \cite{LeBrun1991ExplicitSelfDual} of the scalar curvature of \eqref{eq:Kahler-metric-g}, where in the Type I case $s_g=0$ and in the Type II case $s_g=\xi$. The equations for $W$ follow from the conformal change of Ricci curvature for $g=\xi^2h$. Finally, if we pick a holomorphic coordinate $x+iy$ on $\Sigma$, with $g_{\Sigma}=e^{f(x,y)}\,(dx^2+dy^2)$, then the connection 1-form $\eta$ is determined via
	\begin{equation}\label{eq:deta}
		d\eta=(We^v)_{\xi}\,e^f\,dx\wedge dy+W_x\,dy\wedge d\xi+W_y\,d\xi\wedge dx.
	\end{equation}
	By integrating the (twisted-)Toda equation over $\Sigma$, in both cases there are constants $a,b$ such that $\fint_{\Sigma}We^vd\vol_{\Sigma}=b+\frac{a}{2}\xi$. Indeed, since the degree of the $\bS^1$-bundle is $\ell$, if the period of $\mathcal{K}$ is $\mpp$, we have $\frac{1}{\mpp}\int_{\Sigma} d\eta=\ell$, which is the same as
	\begin{equation}\label{eq:period-constraint}
		\frac{1}{\mpp}\int_{\Sigma}(We^v)_{\xi}d\vol_{\Sigma}=\ell
	\end{equation}
	where $d\vol_{\Sigma}$ is the volume form of $(\Sigma,g_{\Sigma})$.

	\section{Asymptotics and the boundary value problems}
	\label{sec:regularity}

	In this section, we derive the asymptotic behavior of the function $v$
	arising from the reductions \eqref{eq:case-I} and \eqref{eq:case-II} for
	conformally K\"ahler Poincar\'e--Einstein four-manifolds with a cusp end. We first establish the asymptotics of the conformal factor
	$\xi$ along the cusp end, and then use them to determine the corresponding
	behavior of $v$. These asymptotic regimes lead to four natural boundary
	value problems, according to the type of the reduction and the asymptotic
	model of the cusp. In Section \ref{sec:solve-boundary-value-problem}, we will analyze these boundary
	value problems and obtain sharper asymptotics for $v$, which in turn yield
	sharper asymptotics for the Einstein metric $h$ and the K\"ahler metric
	$g$.

	Suppose there is a complete Einstein metric $h$ that is conformal to the K\"ahler metric $g=\xi^2h$, with $\mathcal{K}$ inducing a free $\bS^1$-action and compact K\"ahler reduction $\Sigma$. The manifold $M$ is diffeomorphic to $(0,\infty)\times \mathscr{L}$, with $\mathscr{L}$ being an $\bS^1$-bundle over $\Sigma$. For such $(M,h)$, there are two ends, which we denote by $E_0$ and $E_\infty$. Furthermore, we assume that $E_0$ is PE and $E_\infty$ is an AH cusp, an ACH cusp, or a $\Sigma_{\ttg}$ cusp in the sense of Definition \ref{def:cusps}. We first have
	\begin{lemma}\label{lem:pe-end-compactification}
		The K\"ahler metric $g$ extends smoothly to a K\"ahler metric with boundary on $\overline{M}=[0,\infty)\times\mathscr{L}$. Therefore, the K\"ahler metric $g$ is a conformal compactification for the PE end $E_0$.
	\end{lemma}
	\begin{proof}
		This follows from \cite{LiLiu2025PoincareEinstein}.
	\end{proof}

	Next, we consider the Type I case and the Type II case separately and establish the behavior of $v$ towards the cusp end $E_\infty$.

	\subsection{The Type I case}
	\label{subsec:regularity-Type-I}
	We study the reduction over the cusp end $E_\infty$. Note that a $\Sigma_{\ttg}$ cusp cannot occur in the Type I case since the model metric is not ASD (under both of the orientations). We first consider the asymptotic behavior of the conformal factor $\xi$. Recall that we have arranged that $g=\xi^2h$ and $\xi$ is a moment map for $\mathcal{K}$. Moreover, $\xi$ is proportional to $|d(\mathcal{K}^{\sharp_h})^+|_h^{-1}$. We may suppose $\xi>0$. In the Type I case, since $g$ is K\"ahler scalar-flat, by calculating the conformal change of scalar curvature for $g=\xi^2h$ we have
	\begin{equation}\label{eq:conformal-relation-for-scalar-ASD}
		\Delta_h\xi+2\xi=0.
	\end{equation}
	We have the following Harnack inequality for $\xi$.
	\begin{lemma}\label{lem:Harnack_inequality}
		For any $p\in M$ and any ball of radius one $B_1(p)\subset(M,h)$, there is a uniform constant $C$ such that
		\begin{equation}
			|\nabla_h\xi|_h\leq C\xi.
		\end{equation}
	\end{lemma}
	\begin{proof}
		This is a direct application of the Cheng-Yau gradient estimate \cite{ChengYau1975DifferentialEquations} on $(M,h)$.
	\end{proof}

	Now for any sequence of points $p_j\in M$ that diverges to the cusp end, we consider the balls $B_j:=B_1(p_j)\subset(M,h)$. The sequence of balls $(B_j,h)$ is collapsing in the sense that the injectivity radius converges to zero, but the universal cover $(\widetilde{B}_j,\widetilde{h}_j)$ is non-collapsed. Denote the fundamental group of $B_j$ as $G_j$, which acts on $\widetilde{B}_j$ via deck transformations. Write the hyperbolic metric and complex hyperbolic metric as
	\begin{align}
		h_{\bRH^4}&:=dr^2+e^{-2r}(dx^2+dy^2+dt^2),\\
		h_{\bCH^2}&:=\frac{1}{2}\left(dr^2+e^{-r}(dx^2+dy^2)+e^{-2r}(dt-xdy)^2\right).
	\end{align}
	Since the end is an AH cusp or an ACH cusp, we know $(\widetilde{B}_j,\widetilde{h}_j,G_j)$ converges smoothly in the equivariant Cheeger-Gromov sense to an open domain $(\widetilde{B}_\infty,\widetilde{h}_\infty, G_\infty)$ that embeds into $\bRH^4$ or $\bCH^2$. Here we can isometrically identify the limit as $\widetilde{B}_\infty=\{1<r<3\}$, with the action $G_\infty$ being the standard $\bR^3$-action in the $\bRH^4$ case, and being the standard $\mathcal{H}^3$-action in the $\bCH^2$ case. We may normalize $\xi$ on $\widetilde{B}_1(p_j)$ to $\xi_j:=\xi/\xi(p_j)$. Then because of the Harnack inequality and the elliptic equation \eqref{eq:conformal-relation-for-scalar-ASD}, one has

	\begin{lemma}
		Up to passing to a subsequence, $\xi_j$ converges smoothly to a limit function $\xi_\infty$ on $\widetilde{B}_\infty$ that satisfies $\Delta_{\widetilde{h}_\infty}\xi_\infty+2\xi_\infty=0$. Moreover, $\xi_\infty$ is $G_\infty$-invariant.
	\end{lemma}

	Since $\xi_\infty$ is $G_\infty$-invariant, it is easy to see that
	\begin{itemize}
		\item in the $\bRH^4$ case, there are constants $a,b$ such that
		\[\xi_\infty=a e^r+be^{2r};\]
		\item in the $\bCH^2$ case, there are constants $a,b$ such that
		\[\xi_\infty=(a+br)e^r.\]
	\end{itemize}
	However, since the conformal metric $\widetilde{g}_j:=\xi_j^2\widetilde{h}_j$ is K\"ahler on $\widetilde{B}_j$, the limit metric $\widetilde{g}_\infty:=\xi_\infty^2\widetilde{h}_\infty$ is K\"ahler as well.
	\begin{lemma}
		For the $\bRH^4$ case, $b=0$.
	\end{lemma}
	\begin{proof}
		Note that the K\"ahler metric $g_j:=\xi_j^2h$ on $B_1(p_j)$ is collapsing to an interval. Therefore, the limit K\"ahler metric $\widetilde{g}_\infty$ is K\"ahler scalar flat with 3-dimensional symmetry, so we can apply LeBrun's ansatz for K\"ahler scalar flat metrics. Because the 3-dimensional symmetry is abstractly an action by either $\mathbb{R}^3$ or $\mathcal{H}^3$, we can arrange the Killing fields to be $\partial_{t},\partial_{x},\partial_{y}$, where in the $\mathcal{H}^3$ case $\partial_t$ induces the center $\mathbb{R}$-action. LeBrun's ansatz reduces to
		\[\widetilde{g}_\infty=\widetilde{W}d\xi^2+\widetilde{W}e^{\widetilde{v}}(dx^2+dy^2)+\widetilde{W}^{-1}(dt+\theta)^2\]
		with $(e^{\widetilde{v}})_{\xi\xi}=0$ and $(\widetilde{W}e^{\widetilde{v}})_{\xi\xi}=0$, where $\xi$ is the moment map for $\partial_t$. Here $\theta$ is the 1-form determined by $d\theta=(\widetilde{W}e^{\widetilde{v}})_{\xi} dx\wedge dy+\widetilde{W}_xdy\wedge d\xi+\widetilde{W}_y\,d\xi\wedge dx$. It is easy to see there are constants $A,B,C,D$ such that $e^{\widetilde{v}}=A\xi+B$ and $\widetilde{W}=\frac{C\xi+D}{A\xi+B}$.
		Comparing with $\xi_\infty^2 h_{\bRH^4}$ one can check that they are K\"ahler scalar flat only if $b=0$.
	\end{proof}

    As a direct consequence of the above, we have
	\begin{lemma}\label{lem:gradient_exponential_asd}
		For any $\epsilon\in (0,1)$ and $k\geq0$, there is a constant $C_{\epsilon,k}$, such that on the cusp end
		\[C_{\epsilon,0}^{-1}e^{(1-\epsilon)r}\leq\xi\leq C_{\epsilon,0} e^{(1+\epsilon)r}, \quad |\nabla_h^k\xi|\leq C_{\epsilon,k}e^{(1+\epsilon)r}.\]
	\end{lemma}

	\begin{proof}
		Choose $\tau>0$ small. We claim that for all $q=(r,y)$ with $r$ sufficiently large,
		\begin{equation}\label{eq:xi-ratio-inequality-asd}
			e^{(1-\epsilon)\tau}\le \frac{\xi(r+\tau,y)}{\xi(r,y)}\le e^{(1+\epsilon)\tau}.
		\end{equation}
		Otherwise, there exists a sequence $q_j=(r_j,y_j)$ with $r_j\to\infty$ such that one of the inequalities in \eqref{eq:xi-ratio-inequality-asd} fails at $q_j$. After passing to a subsequence and recentering at $q_j$, the normalized functions
		$
		\xi_j:=\frac{\xi}{\xi(q_j)}
		$
		converge uniformly on a fixed ball to $e^s$, where $s=0$ at the basepoint. Since $(r_j+\tau,y_j)$ converges to the point $s=\tau$ in the limit chart, we obtain
		$
		\xi_j(r_j+\tau,y_j)\to e^\tau,
		$
		contradicting the failure of $\eqref{eq:xi-ratio-inequality-asd}$. Hence $\eqref{eq:xi-ratio-inequality-asd}$ holds.

		Now write $r=r_0+N\tau$, where $r_0$ lies in a fixed compact interval and $N\in\mathbb N$. Iterating $\eqref{eq:xi-ratio-inequality-asd}$ yields
		\[
		e^{(1-\epsilon)(r-r_0)}\xi(r_0,y)=e^{(1-\epsilon)N\tau}\xi(r_0,y)\le \xi(r,y)\le e^{(1+\epsilon)N\tau}\xi(r_0,y)=e^{(1+\epsilon)(r-r_0)}\xi(r_0,y).
		\]
		Higher derivative estimates follow from Lemma \ref{lem:Harnack_inequality} and \eqref{eq:conformal-relation-for-scalar-ASD}.
	\end{proof}

	Since the cusp end is exponentially asymptotic to the corresponding model cusp metric, we may apply Fredholm theory to \eqref{eq:conformal-relation-for-scalar-ASD}.

	\begin{lemma}\label{lem:asymptotic-for-xi-ASD}
		There exists $\delta_1\in(0,\delta_0)$ such that, for every $k\ge 0$,
		\begin{itemize}
			\item in the AH cusp case, there exists a constant $a>0$ such that
			\[
			|\nabla_{h_0}^k(\xi-ae^r)|_{h_0}=O(e^{(1-\delta_1)r});
			\]
			\item in the ACH cusp case, there exist constants $a\in\mathbb R$ and $b\geq 0$ such that
			\[
			|\nabla_{h_0}^k(\xi-(a+br)e^r)|_{h_0}=O(e^{(1-\delta_1)r}).
			\]
		\end{itemize}
	\end{lemma}
	We will show $b=0$ for the second case in the proof of Proposition \ref{prop:asymptotic-v-cusp}.

	\begin{proof}
		Since $h$ is exponentially asymptotic to $h_0$, the equation \eqref{eq:conformal-relation-for-scalar-ASD} may be rewritten as
		\begin{equation}\label{eq:xi-linearized-cusp}
			\Delta_{h_0}\xi+2\xi = F,
		\end{equation}
		where
		\[
		F=
		O(e^{-\delta_0 r})\cdot \nabla_{h_0}^2\xi
		+
		O(e^{-\delta_0 r})\cdot \nabla_{h_0}\xi.
		\]
		By Lemma \ref{lem:gradient_exponential_asd}, we have $F=O(e^{(1-\delta_1)r})$ for any $\delta_1\in(0,\delta_0]$.

		We write
		\[
		\xi=\sum_j \alpha_j(r)\Theta_j(y), \qquad F=\sum_j \beta_j(r)\Theta_j(y),
	\]
	where $r\in(R_0,\infty)$ and $y\in\mathscr L$. This reduces \eqref{eq:xi-linearized-cusp} to a family of ODEs for the coefficients $\alpha_j(r)$. Solving these ODEs and using Lemma \ref{lem:gradient_exponential_asd} shows that for $\delta_1\in (0,\delta_0]$ that avoids finitely many numbers,
		\[
		|\nabla^k_{h_0}(\xi-\hat\xi)|_{h_0}=O(e^{(1-\delta_1)r})
		\]
		for every $k\ge 0$, where $\hat\xi=ae^r$ in the AH cusp case and $\hat\xi=(a+br)e^r$ in the ACH cusp case.
	\end{proof}

	Since $\xi$ is a defining function for the PE end $E_0$, its range on $M$ is $(0,\infty)$. Hence the reduction is defined globally:
	\begin{align*}
		g&=Wd\xi^2+W^{-1}\eta^2+We^vg_{T^2},\\
		h&=\xi^{-2}(Wd\xi^2+W^{-1}\eta^2+We^vg_{T^2}),
	\end{align*}
	where $v$ and $W$ satisfy
	\begin{equation}\label{eq:toda_equation_asd_uniqueness_proof}
		(e^v)_{\xi\xi}+\Delta_{T^2}v=0
		\qquad
		W=1-\frac12\xi v_\xi.
	\end{equation}

	\begin{proposition}\label{prop:asymptotic-v-cusp}
		On the cusp end $E_\infty$,
		\begin{itemize}
			\item in the AH cusp case,
			$
			|v|<C;
			$
			\item in the ACH cusp case,
			$
			|v-\log\xi|<C.
			$
		\end{itemize}
	\end{proposition}

	\begin{proof}
		Since $\xi^2W^{-1}=|d\xi|_h^2$, the asymptotics of $W$ follow from those of $\xi$ in Lemma \ref{lem:asymptotic-for-xi-ASD} together with the exponential convergence $h\to h_0$.

		In the AH cusp case, Lemma \ref{lem:asymptotic-for-xi-ASD} gives
		\[
		\xi=ae^r+O(e^{(1-\delta_1)r}), \qquad \partial_r\xi=ae^r+O(e^{(1-\delta_1)r}).
		\]
		Since $h$ is exponentially asymptotic to the AH cusp model $h_0$, we have $|dr|_h^2=1+O(e^{-\delta_1 r})$. Hence
		\[
		|d\xi|_h^2=\xi^2\bigl(1+O(e^{-\delta_1 r})\bigr),
		\qquad
		W=\frac{\xi^2}{|d\xi|_h^2}=1+O(e^{-\delta_1 r}).
		\]
		Using $W=1-\frac12\xi v_\xi$, we obtain
		$
		v_\xi=\frac{2(1-W)}{\xi}=O(\xi^{-1-\delta_1}),
		$
		and integrating in $\xi$ gives $v=c_0+O(\xi^{-\delta_1})$, where $c_0$ is a smooth function on $T^2$.

		In the ACH cusp case, Lemma \ref{lem:asymptotic-for-xi-ASD} gives 
		\[
		\xi=(a+br)e^r+O(e^{(1-\delta_1)r}), \qquad \partial_r\xi=(a+b+br)e^r+O(e^{(1-\delta_1)r}).
		\]
		Since $h$ is exponentially asymptotic to the ACH cusp model $h_0$, we have $|dr|_h^2=2+O(e^{-\delta_1 r})$. It follows that, after decreasing $\delta_1$ if necessary to absorb polynomial factors in $r$, we have
		\[
		|d\xi|_h^2
		=
		2(a+b+br)^2e^{2r}+O(e^{(2-\delta_1)r})
		=
		2\Bigl(1+\frac{b}{a+br}\Bigr)^2\xi^2+O(e^{(2-\delta_1)r}),
		\]
		and hence
		\[
		W=\frac{\xi^2}{|d\xi|_h^2}
		=
		\frac12\Bigl(1+\frac{b}{a+br}\Bigr)^{-2}+O(e^{-\delta_1 r}).
		\]

		\begin{equation}\label{eq:Wev_ev_ratio1}
			\frac{\int_{T^2}We^v\dvol_{T^2}}{\int_{T^2}e^v\dvol_{T^2}}
			=
			\frac12\Bigl(1+\frac{b}{a+br}\Bigr)^{-2}+O(e^{-\delta_1 r}).
		\end{equation}

		On the other hand, integrating $(e^v)_{\xi\xi}+\Delta_{T^2}v=0$ over $T^2$ gives
		$
		\int_{T^2}e^v=\lambda\xi+\mu, \,\lambda>0.
		$
		Moreover, using $W=1-\frac12\xi v_\xi$, we have
		$
		We^v=e^v-\frac12\xi(e^v)_\xi,
		$
		and hence
		$
		\int_{T^2}We^v=\frac12\lambda\xi+\mu.
		$
		Therefore
		\begin{equation}\label{eq:Wev_ev_ratio2}
			\frac{\int_{T^2}We^v\dvol_{T^2}}{\int_{T^2}e^v\dvol_{T^2}}
			=
			\frac{\frac12\lambda\xi+\mu}{\lambda\xi+\mu}
			=
			\frac12+O(\xi^{-1})
			=
			\frac12+O((re^r)^{-1}).
		\end{equation}

		Comparing \eqref{eq:Wev_ev_ratio1} and \eqref{eq:Wev_ev_ratio2} gives $b=0$. Therefore
		$
		W=\frac12+O(e^{-\delta_1 r}).
		$
		Using again $W=1-\frac12\xi v_\xi$, we obtain
		$
		v_\xi-\frac1\xi=\frac{1-2W}{\xi}=O(\xi^{-1-\delta_1}),
		$
		and integrating gives
		$
		v=\log\xi+c_0+O(\xi^{-\delta_1}),
		$
		where $c_0$ is a smooth function on $T^2$.
	\end{proof}

	\subsection{The Type II case}
	\label{subsec:regularity-Type-II}

	Now we consider the Type II case. Recall that in the Type II case we arrange $g=s_g^2h$ with $\xi=s_g$ being the scalar curvature of the K\"ahler extremal metric $g$, which is either strictly positive or strictly negative in $M$. Moreover,
	\begin{equation}\label{eq:xi-and-Weyl}
		|\xi|=(2\sqrt{6}|W^+_h|_h)^{\frac{1}{3}}.
	\end{equation}
	It is impossible for $M$ to have both a PE end and an AH cusp end; otherwise, by \eqref{eq:xi-and-Weyl}, $\xi \to 0$ at both ends, so $\xi$ has an interior critical point, contradicting the freeness of the $\bS^1$-action.

	We first consider the case that $E_\infty$ is an ACH cusp. Then the orientation induced by the complex structure agrees with the positive orientation of $\bCH^2$; otherwise, \eqref{eq:xi-and-Weyl} again gives $\xi \to 0$ on both ends, yielding the same contradiction.

	Since $E_\infty$ is an ACH cusp with the positive orientation of $\bCH^2$, $W_h^+$ approaches the self-dual Weyl tensor of positively oriented $\bCH^2$, which, in an orthogonal frame of $\Lambda^+$, is $\mathrm{diag}(-2,1,1)$. Hence
	$|\xi|=(2\sqrt{6}|W_h^+|)^{1/3}\to \sqrt[3]{12}$
	as we approach infinity of $E_\infty$. Since the conformal factor tends to a constant, the K\"ahler metric $g$ is also asymptotic, up to scaling, to the standard complex hyperbolic cusp metric. As $\bCH^2$ has scalar curvature $-12$, it follows that $\xi<0$ on $E_\infty$, thus $\xi<0$ on $M$. For simplicity we shall set
	\[ \xi_\ast:=-\sqrt[3]{12}.\]

	The conformal change of scalar curvature associated to $g=\xi^2h$ is
	\begin{equation}\label{eq:conformal-change-of-scalar-Type-II}
		6\Delta_h\xi+\xi(12+\xi^3)=0,
	\end{equation}
	which can be written as
	\begin{equation}\label{eq:modified-conformal-change-of-scalar-Type-II}
		6\Delta_h(\xi- \xi_\ast)+\xi(\xi- \xi_\ast)(\xi^2+ \xi_\ast\xi+ \xi_\ast^2)=0.
	\end{equation}

	Now we identify the end $E_\infty$ as $(R_0,\infty)\times\NN_{\ell}$ via the diffeomorphism in Definition \ref{def:cusps}.

	\begin{lemma}\label{lem:gradient_exponential_typeii}
		We have $\xi>\xi_\ast$ on $M$, and for any $\epsilon\in (0,1)$ and $k\geq0$, there is a constant $C_{\epsilon,k}$, such that on the cusp end
		\[C_{\epsilon,0}^{-1}e^{-(1+\epsilon)r}\leq \xi-\xi_\ast\leq C_{\epsilon,0} e^{-(1-\epsilon)r}, \quad \qquad |\nabla_h^k(\xi-\xi_\ast)|\leq C_{\epsilon,k}e^{-(1-\epsilon)r}.\]
	\end{lemma}

	\begin{proof}
		Set $u:=\xi-\xi_\ast$. Then $u$ satisfies
		\begin{equation}\label{eq:modified-conformal-change-of-scalar-Type-II-in-u}
			\Delta_h u-qu=0,
			\qquad
			q:=-\frac{1}{6}\xi(\xi^2+\xi_\ast\xi+\xi_\ast^2)\geq 0.
		\end{equation}
		Since $\xi\to \xi_\ast$ on the ACH cusp end and $\xi_\ast^3=-12$, we have $q\to 6$.

		We first show $u>0$ on $M$. Since $u\to -\xi_\ast>0$ on one end and $u\to 0$ on the other, the maximum principle applied on an exhaustion of $M$ gives $u\ge0$. If $u=0$ at some interior point, then $u$ attains its minimum there, so by the strong maximum principle $u$ is constant, contradicting the end behavior. Hence $u>0$ on $M$.

		For the model complex hyperbolic metric $h_0$, a direct computation shows that for any $\alpha\in\mathbb R$,
		\begin{equation}\label{eq:Laplace_minus_6_ACH_cusp}
			(\Delta_{h_{0}}-6)e^{-\alpha r}
			=2(\alpha-1)(\alpha+3)e^{-\alpha r}.
		\end{equation}

		Since $h\to h_0$ and $q\to 6$ on the cusp end, we can find $R$ such that on $\{r\ge R\}$,
		\[
		(\Delta_h-q)e^{-(1-\epsilon)r}<0,\qquad (\Delta_h-q)e^{-(1+\epsilon)r}>0.
		\]
		Choose $c,C>0$ such that $ce^{-(1+\epsilon)R}\le u\le Ce^{-(1-\epsilon)R}$ on $\{r=R\}$. Since $u\rightarrow 0$ on the cusp end, applying the maximum principle to $u-Ce^{-(1-\epsilon)r}$ and $u-ce^{-(1+\epsilon)r}$ on $\{r\ge R\}$ yields
		\[
		c\,e^{-(1+\epsilon)r}\le u\le C\,e^{-(1-\epsilon)r},
		\]
		which proves the first claim when $k=0$. The derivative estimates follow from standard elliptic theory applied to $\Delta_h u - qu=0$.
	\end{proof}

	Therefore $\xi=s_g$ takes values in $(\xi_\ast,0)$ on $M$.

	\begin{lemma}\label{lem:asymptotic-for-xi-Type-II}
		There exists $\delta_1\in (0,\delta_0)$ and a constant $a$ such that
		\[|\nabla_{h_0}^k(\xi- \xi_\ast-ae^{-r})|_{h_0}=O(e^{-(1+\delta_1)r}).\]
	\end{lemma}
	\begin{proof}
		From \eqref{eq:modified-conformal-change-of-scalar-Type-II} or \eqref{eq:modified-conformal-change-of-scalar-Type-II-in-u}, we have
		\begin{align*}
			&(\Delta_{h_0}-6)(\xi- \xi_\ast)\\
			=&O(e^{-\delta_0 r})\cdot\nabla_{h_0}(\xi- \xi_\ast)+O(e^{-\delta_0 r})\cdot\nabla_{h_0}^2(\xi- \xi_\ast)+O(e^{-(1-\epsilon)r})(\xi- \xi_\ast).
		\end{align*}
		With the pointwise bounds of Lemma~\ref{lem:gradient_exponential_typeii} in hand,
		the lemma follows from a standard application of Fredholm theory for
		$\Delta_{h_0}-6$ on the complex hyperbolic cusp.
	\end{proof}

	So we can apply the reduction over the entire $M$ to write
	\begin{align*}
		g&=Wd\xi^2+W^{-1}\eta^2+We^vg_{T^2},\\
		h&=\xi^{-2}(Wd\xi^2+W^{-1}\eta^2+We^vg_{T^2}),
	\end{align*}
	with $\xi\in(\xi_\ast,0)$, and $v,W$ satisfy the twisted-Toda system
	\begin{align}\label{eq:twisted-toda-system}
		(e^v)_{\xi\xi}+\Delta_{T^2}v=-\xi We^v,\qquad
		W=\frac{12-6\xi v_\xi}{12+\xi^3}.
	\end{align}

	With $W^{-1}=\xi^{-2}|d\xi|^2_h$, this shows that
	\[|\nabla_{h_0}^k(W^{-1}-2 \xi_\ast^{-2}a^2e^{-2r})|_{h_0}=O(e^{-(2+\delta_2)r})\]
	where $\delta_2=\min\{1,\delta_1\}$, which implies
	\[|W-\frac{1}{2} \xi_\ast^2(\xi- \xi_\ast)^{-2}|=O((\xi- \xi_\ast)^{-2+\delta_2}).\]
	Together with $W=\frac{12-6\xi v_\xi}{12+\xi^3}$, we have $v_\xi=3(\xi- \xi_\ast)^{-1}+\frac{2}{\xi_\ast}+O((\xi- \xi_\ast)^{-1+\delta_2})$. By integration, it follows that
	\begin{proposition}
		On the ACH cusp end $E_\infty$,
		\[|v-3\log(\xi- \xi_\ast)|<C.\]
	\end{proposition}

	Next we consider the situation that $E_\infty$ is a $\Sigma_{\ttg}$ cusp. For the same reason as in the ACH cusp case, we know that the range of $\xi=s_g$ on $M$ is $(\xi_\ast,0)$. We have Lemma \ref{lem:gradient_exponential_typeii} as well: note that for the model $\Sigma_{\ttg}$
	cusp metric $h_0$ (both trivial and non-trivial case), instead of \eqref{eq:Laplace_minus_6_ACH_cusp}, we have
	\begin{equation}
		(\Delta_{h_{0}}-6)e^{-\alpha r}
		=3(\alpha-1)(\alpha+2)e^{-\alpha r}.
	\end{equation}

	We also have
	\[|\nabla_{h_0}^k(\xi-\xi_\ast-ae^{-r})|_{h_0}=O(e^{-(1+\delta_1)r})\]
	as in Lemma \ref{lem:asymptotic-for-xi-Type-II}, which gives that
		\[|\nabla_{h_0}^k(W^{-1}-3\xi_\ast^{-2}a^2e^{-2r})|_{h_0}=O(e^{-(2+\delta_2)r}),\]
		where $\delta_2=\min\{1,\delta_1\}$. This implies
	\[|W-\frac{1}{3} \xi_\ast^2(\xi- \xi_\ast)^{-2}|=O((\xi- \xi_\ast)^{-2+\delta_2}).\]
	Recall $W=\frac{12-6\xi v_\xi}{12+\xi^3}$. By integration, it follows that
	\begin{proposition}
		On the $\Sigma_{\ttg}$ cusp end $E_\infty$,
		\[|v-2\log(\xi-\xi_\ast)|<C.\]
	\end{proposition}

	\subsection{The boundary value problems}
	\label{subsec:the-boundary-value-problem}

	Finally, we formulate our boundary value problems. Based on our analysis of the asymptotic behavior of solutions $v$ to the (twisted-)Toda equation in Sections \ref{subsec:regularity-Type-I} and \ref{subsec:regularity-Type-II}, we are led to the following four natural boundary value problems with smooth Dirichlet boundary data $\varphi$.

  \begin{align}
  	&\label{eq:bvp1}
  	\tag{BVP1}
  	\left\{
  	\begin{aligned}
  		&(e^v)_{\xi\xi}+\Delta_{T^2}v = 0
  		\qquad \text{in } (0,\infty)\times T^2,\\
  		&v|_{\{0\}\times T^2} = \varphi,\\
  		&v \in L^\infty\bigl((0,\infty)\times T^2\bigr).
  	\end{aligned}
  	\right.\\
  	\notag\\
  	&\label{eq:bvp2}
  	\tag{BVP2}
  	\left\{
  	\begin{aligned}
  		&(e^v)_{\xi\xi}+\Delta_{T^2}v = 0
  		\qquad \text{in } (0,\infty)\times T^2,\\
  		&v|_{\{0\}\times T^2} = \varphi,\\
  		&v-\log(\xi+1) \in L^\infty\bigl((0,\infty)\times T^2\bigr).
  	\end{aligned}
  	\right.\\
  	\notag\\
  	&\label{eq:bvp3}
  	\tag{BVP3}
  	\left\{
  	\begin{aligned}
  		&(e^v)_{\xi\xi}+\Delta_{T^2}v = -\xi\frac{12-6\xi v_\xi}{12+\xi^3}e^v
  		\qquad \text{in } (\xi_\ast,0)\times T^2,\\
  		&v|_{\{0\}\times T^2} = \varphi,\\
  		&v-3\log(\xi- \xi_\ast) \in L^\infty\bigl((\xi_\ast,0)\times T^2\bigr).
  	\end{aligned}
  	\right.\\
  	\notag\\
  	&\label{eq:bvp4}
  	\tag{BVP4}
  	\left\{
  	\begin{aligned}
  		&(e^v)_{\xi\xi}+\Delta_{\Sigma_{\ttg}}v =-2 -\xi\frac{12-6\xi v_\xi}{12+\xi^3}e^v
  		\qquad \text{in } (\xi_\ast,0)\times \Sigma_{\ttg},\\
  		&v|_{\{0\}\times \Sigma_{\ttg}} = \varphi,\\
  		&v-2\log(\xi- \xi_\ast) \in L^\infty\bigl((\xi_\ast,0)\times \Sigma_{\ttg}\bigr).
  	\end{aligned}
  	\right.
  \end{align}

	Suppose that there exists a smooth solution $v$, and suppose moreover that the corresponding functions $W$, determined by $v$ through the reductions \eqref{eq:case-I}--\eqref{eq:case-II}, are positive. As explained in Section \ref{sec:background}, $\int_{\Sigma}(We^v)_\xi d\vol_{\Sigma}$ is a constant. For given $\ell$, by choosing the period $\mpp$ correctly we can arrange \eqref{eq:period-constraint}. Then geometrically these solutions give rise to Einstein metrics $h$ on $M=(0,\infty)\times\mathscr{L}$ with the following asymptotic behavior:
	\begin{itemize}
		\item BVP1: $h$ is Type I, $E_0$ is a PE end, and $E_\infty$ is an AH cusp end;
		\item BVP2: $h$ is Type I, $E_0$ is a PE end, and $E_\infty$ is an ACH cusp end;
		\item BVP3: $h$ is Type II, $E_0$ is a PE end, and $E_\infty$ is an ACH cusp end;
		\item BVP4: $h$ is Type II, $E_0$ is a PE end, and $E_\infty$ is a $\Sigma_{\ttg}$ cusp end.
	\end{itemize}
	Moreover, there is the K\"ahler metric $g$ that extends to a smooth complete K\"ahler metric with boundary $\bar g$ on $\overline{M}=[0,\infty)\times\mathscr{L}$. The quotient of $\bar g|_{\{0\}\times\mathscr{L}}$ by $\bS^1$ is $(\Sigma,e^{\varphi}g_{\Sigma})$, and hence is determined by the boundary data $\varphi$. Conversely, by Section \ref{sec:background} and Sections \ref{subsec:regularity-Type-I}--\ref{subsec:regularity-Type-II}, we know any conformally K\"ahler PE metric with a cusp on $(0,\infty)\times\mathscr{L}$, whose $\mathcal{K}$ induces an $\bS^1$-action, arises from solutions to these boundary value problems.

	In Sections \ref{subsec:hyperbolic-cusp}--\ref{subsec:cuspxsurface-typeII}, we shall prove the following results.
	\begin{theorem}\label{thm:bvp-existence-uniqueness}
		The boundary value problems \eqref{eq:bvp1}, \eqref{eq:bvp3} and \eqref{eq:bvp4} have unique solutions, and \eqref{eq:bvp2} has a unique one-parameter family of solutions.
	\end{theorem}

	For $\Sigma=T^2$ or $\Sigma_{\ttg}$, set
	\[
	\bar\varphi=\log(\fint_{\Sigma} e^\varphi \, d\vol_{\Sigma}).
	\]
	We will also prove that all solutions are asymptotic to model solutions:

	\begin{theorem}\label{thm:bvp-model-asymptotics}
		Let $v$ be a solution of one of the boundary value problems
		\eqref{eq:bvp1}--\eqref{eq:bvp4}. Then $v$ is asymptotic to the corresponding
		model solution $v_{\mathrm{mod}}$ with the following rates:
		\begin{enumerate}
		\item For \eqref{eq:bvp1}, as $\xi\to\infty$,
		\[
		|v-v_{\mathrm{mod}}|=O(e^{-\delta \xi}),
		\qquad
		v_{\mathrm{mod}}=\bar\varphi .
		\]
		
		\item For \eqref{eq:bvp2}, there exists $a>0$ such that, as $\xi\to\infty$,
		\[
		|v-v_{\mathrm{mod}}|=O(e^{-\delta\sqrt{\xi}}),
		\qquad
		v_{\mathrm{mod}}=\log(e^{\bar\varphi}+a\xi).
		\]
		
		\item For \eqref{eq:bvp3}, as $\xi\to\xi_\ast$,
		\[
		|v-v_{\mathrm{mod}}|
		=
		O\!\left(e^{-\delta(\xi-\xi_\ast)^{-1/2}}\right),
		\qquad
		v_{\mathrm{mod}}
		=
		\bar\varphi
		+3\log\left(1-\frac{\xi}{\xi_\ast}\right)
		+\log\left(1+\frac{\xi}{\xi_\ast}\right).
		\]
		
		\item For \eqref{eq:bvp4}, as $\xi\to\xi_\ast$,
		\[
		|v-v_{\mathrm{mod}}|
		=
		O\!\left((\xi-\xi_\ast)^\delta\right),
		\qquad
		v_{\mathrm{mod}}
		=
		2\log(\xi-\xi_\ast)
		+
		\log\left(
		\frac{1}{3}
		-\frac{\xi_\ast^2}{18}\xi
		+\frac{a}{24}(\xi_\ast^2-\xi^2)
		\right),
		\]
		where
		\[
		a
		=
		\frac{2}{\xi_\ast}
		\left(
		\frac{1}{3}\xi_\ast^2-e^{\bar\varphi}
		\right).
		\]
		\end{enumerate}
		Here $\delta>0$ depends on $v$ and $(\Sigma,g_{\Sigma})$.
	\end{theorem}

	Here, for \eqref{eq:bvp2}, each $a>0$ and boundary value $\varphi$ determine a unique solution $v_{a,\varphi}$ to \eqref{eq:bvp2} with the above asymptotic. Given the solution $v_{1,\varphi}$, the function $v_{1,\varphi}(a\xi,\cdot)-2\log a$ is also a solution, and is exactly $v_{a^{-1},\varphi-2\log a}$. These solutions give rise to the same Einstein metric up to isometry.

	The decay of $v$ shows that the cusp ends arising from \eqref{eq:bvp1}--\eqref{eq:bvp3} are asymptotic to their model cusp metrics at a super-exponential rate, while the $\Sigma_{\ttg}$ cusp end arising from \eqref{eq:bvp4} has exponential decay. Note that the $\Sigma_{\ttg}$ cusp has degree zero if and only if $a=0$, i.e.,
	$
	\fint_{\Sigma_{\ttg}} e^\varphi \,d\vol_{\Sigma_{\ttg}}=\frac{1}{3}\xi_\ast^2 .
	$

	Summarizing this section, we have

	\begin{theorem}\label{thm:bvp-classification}
		Suppose $h$ is a complete Einstein metric with one PE end and one AH/ACH/$\Sigma_{\ttg}$ cusp end on $(0,\infty)\times\mathscr{L}$, which is conformally K\"ahler and the Killing field $\mathcal{K}$ induces a free $\bS^1$-action. Then it arises from a solution of one of the boundary value problems \eqref{eq:bvp1}--\eqref{eq:bvp4} via the ansatz \eqref{eq:Einstein-metric-h}. Here $W$ is determined by \eqref{eq:case-I} or \eqref{eq:case-II}, and $\eta$ is determined by \eqref{eq:deta}. The corresponding potential $v$ has the asymptotic behavior described in Theorem \ref{thm:bvp-model-asymptotics}.
	\end{theorem}
	In particular, the weak assumption in Definition \ref{def:cusps} that $h$ is asymptotic to a model cusp metric $h_0$ implies the much stronger asymptotic behavior described above.

	\section{Model solutions and their geometries}
	\label{sec:model-geometry}

	Before solving the boundary value problems \eqref{eq:bvp1}--\eqref{eq:bvp4}, we first study the simplest solutions of the (twisted-)Toda equation, namely those depending only on $\xi$. Part of our solutions here, namely the subclass for which the connection one-form $\eta$ is flat, in the sense that $d\eta=0$, can be identified with the so-called Euclidean topological black holes \cite{Birmingham1999TopologicalBHAdS}.

	\subsection{The case of Type I}
	\label{subsec:model-solutions-Type-I}

	Solutions to \eqref{eq:case-I} that only depend on $\xi$ take the form of
	\[
	We^v=b+\frac a2\xi,
	\qquad
	e^v=b+a\xi.
	\]
	The corresponding metric is
	\begin{equation}\label{eq:typeI-metric}
		h_{\mathrm{mod}}
		=
		\frac1{\xi^{2}}\left(
		\frac{b+\frac a2\xi}{b+a\xi}\,d\xi^{2}
		+
		\frac{b+a\xi}{b+\frac a2\xi}\,\eta^{2}
		+
		\left(b+\frac a2\xi\right)g_{T^{2}}
		\right),
		\qquad
		d\eta=\frac a2\,d\vol_{T^2}.
	\end{equation}
	Here $h_{\mathrm{mod}}$ is Riemannian precisely when
	\begin{equation}\label{eq:typeI-positivity}
		b+\frac a2\xi>0,
		\qquad
		b+a\xi>0.
	\end{equation}
	Let $I\subset (0,\infty)$ be the maximal interval on which \eqref{eq:typeI-positivity} holds.
	We now determine $I$ and describe the geometry at the endpoints of $I$. The maximal interval $I\subset(-\infty,0)$ such that \eqref{eq:typeI-positivity} can be determined similarly. For an endpoint $\hat\xi$ of $I$, we say the $\hat\xi$-end is
	\begin{itemize}
		\item \emph{conical} if the Einstein metric $h$ at $\hat\xi$ completes to a metric with cone angle along $T^2$;
		\item \emph{$\frac{2}{3}$-horn-type} if reparametrizing $\xi$ as a suitable $\zeta=\zeta(\xi)>0$ with $\zeta(\hat\xi)=0$,
		\[h\sim d\zeta^2+c_1\zeta^{-2/3}\eta^2+c_2\zeta^{2/3}g_{T^2},\]
		and \emph{$\frac{4}{3}$-horn-type} if
		\[h\sim d\zeta^2+c_1\zeta^{-2/3}\eta^2+c_2\zeta^{4/3}g_{T^2},\]
		for some constants $c_1,c_2$.
	\end{itemize}
	We summarize the maximal Riemannian intervals and the corresponding
    endpoint geometries in the following table. 

	\begin{table}[htbp]
		\centering
		\small
		\renewcommand{\arraystretch}{1.18}
		\setlength{\tabcolsep}{5pt}

		\begin{tabularx}{\textwidth}{C{0.9cm} C{2.2cm} C{4.0cm} Y}
			\toprule
			\textbf{Case} & \textbf{Condition} & \textbf{Maximal interval $I$} & \textbf{Behavior at endpoints of $I$} \\
			\midrule

			1
			& $a=0,\ b>0$
			& $(0,\infty)$
			& \makecell[l]{$\infty$-end AH cusp.\\
				Hyperbolic cusp metric.} \\
			\addlinespace

			2
			& $a>0,\ b>0$
			& $(0,\infty)$
			& $\infty$-end ACH cusp. \\
			\addlinespace

			3
			& $a<0,\ b>0$
			& \makecell[c]{$(0,-\frac{b}{a})$}
			& $-\frac{b}{a}$-end conical. \\
			\addlinespace

			4
			& $a>0,\ b<0$
			& \makecell[c]{$(-\frac{2b}{a},\infty)$}
			& \makecell[l]{$-\frac{2b}{a}$-end $\frac{2}{3}$-horn-type;\\
				$\infty$-end ACH cusp.} \\
			\addlinespace

			5
			& $b=0,\ a>0$
			& $(0,\infty)$
			& \makecell[l]{$\infty$-end ACH cusp.\\
				Complex hyperbolic cusp metric.} \\
			\bottomrule
		\end{tabularx}
		\medskip
		\caption{Maximal Riemannian intervals for the Type I family. In cases 1–3, the $0$-end is PE, and in case 5, the $0$-end is ACH.  Here, ACH refers to \emph{asymptotically complex hyperbolic} in the sense of \cite{BiquardRollin2009Wormholes}; An ACH end is a complete, expanding end, distinct from an ACH cusp end.}
		\label{tab:typeI-cases}
	\end{table}

	\subsection{The case of Type II with base \texorpdfstring{$T^2$}{T2}}
	\label{subsec:model-solution-Type-II}

	We now turn to the twisted-Toda equation with $\Sigma=T^2$. Solutions to \eqref{eq:case-II} that only depend on $\xi$ are
	\[
	We^v=b+\frac a2\xi,
	\qquad
	e^v=b+a\xi-\frac b6\xi^3-\frac a{24}\xi^4.
	\]
	Set
	\[
	P(\xi):=b+a\xi-\frac b6\xi^3-\frac a{24}\xi^4.
	\]
	Then
	\begin{equation}\label{eq:typeII-metric}
		h_{\mathrm{mod}}
		=
		\frac1{\xi^2}\left(
		\frac{b+\frac a2\xi}{P(\xi)}\,d\xi^2
		+
		\frac{P(\xi)}{b+\frac a2\xi}\,\eta^2
		+
		\left(b+\frac a2\xi\right)g_{T^2}
		\right),
		\qquad
		d\eta=\frac a2\,d\vol_{T^2}.
	\end{equation}
	Hence $h_{\mathrm{mod}}$ is Riemannian precisely when
	\begin{equation}\label{eq:typeII-positivity}
		b+\frac a2\xi>0,
		\qquad
		P(\xi)>0.
	\end{equation}
	Set $\underline{\xi}:=-\frac{2b}{a}$. We make the following convention.
	\begin{itemize}
		\item When $a>0,b\neq0$, the largest root of $P$ is positive, denoted as $\xi_+$.
		\item When $a<0,b>0$, $P$ has a unique root in $(0,\underline{\xi})$ denoted as $\xi_+$.
		\item When $3a^3>2b^3>0$, $P$ has a unique root in $(\underline{\xi},0)$ denoted as $\xi_-$.
		\item When $2b^3<3a^3<0$, $P$ has a unique root in $(-\infty,\underline{\xi})$ denoted as $\xi_-$.
	\end{itemize}

	We summarize the maximal Riemannian intervals and the corresponding endpoint geometries in the following Table \ref{tab:model-cases}. 
	\begin{table}[htbp]
		\centering
		\small
		\renewcommand{\arraystretch}{1.18}
		\setlength{\tabcolsep}{5pt}

		\begin{tabularx}{\textwidth}{C{0.9cm} C{2.2cm} C{4.6cm} Y}
			\toprule
			\textbf{Case} & \textbf{Condition} & \textbf{Maximal interval $I$} & \textbf{Behavior at endpoints of $I$} \\
			\midrule

			1
			& $a=0,\ b>0$
			& \makecell[c]{$(0,\sqrt[3]{6})$, $(-\infty,0)$}
			& \makecell[l]{$\sqrt[3]{6}$-end conical; $-\infty$-end $\frac{4}{3}$-horn-type.} \\
			\addlinespace

			2
			& $a>0,\ b>0$
			& \makecell[c]{$(0,\xi_+)$, and\\
				$
				\left\{
				\begin{array}{ll}
					(\xi_-,0), & \text{if }2b^3<3a^3,\\
					(\underline{\xi},0), & \text{if }2b^3\ge 3a^3
				\end{array}
				\right.
				$}
			& \makecell[l]{ $\xi_+$-end conical. On the negative side: \\ conical if $2b^3<3a^3$; \\ ACH cusp if $2b^3=3a^3$; \\ $\frac{2}{3}$-horn-type if $2b^3>3a^3$.} \\
			\addlinespace

			3
			& $a<0,\ b>0$
			& $(0,\xi_+)$, $(-\infty,0)$
			& \makecell[l]{ $\xi_+$-end conical; $-\infty$-end $\frac{2}{3}$-horn-type.} \\
			\addlinespace

			4
			& $a>0,\ b<0$
			& $(\underline{\xi},\xi_+)$
			& $\underline{\xi}$-end $\frac{2}{3}$-horn-type; $\xi_+$-end conical. \\
			\addlinespace

			5
			& $a<0,\ b<0$
			& \makecell[c]{
				$
				\left\{
				\begin{array}{ll}
					(-\infty,\underline{\xi}), & \text{if }2b^3\ge 3a^3,\\
					(-\infty,\xi_-), & \text{if }2b^3<3a^3
				\end{array}
				\right.
				$}
			& \makecell[l]{$-\infty$-end $\frac{2}{3}$-horn-type. For the other end: \\ $\frac{2}{3}$-horn-type if $2b^3>3a^3$; \\ ACH cusp if $2b^3=3a^3$; \\ conical if $2b^3<3a^3$.} \\
			\addlinespace

			6
			& $b=0,\ a>0$
			& \makecell[c]{$(0,\sqrt[3]{24})$}
			& $\sqrt[3]{24}$-end conical. \\
			\addlinespace

			7
			& $b=0,\ a<0$
			& $(-\infty,0)$
			& $-\infty$-end $\frac{2}{3}$-horn-type. \\
			\bottomrule
		\end{tabularx}
        \medskip
		\caption{Maximal Riemannian intervals for the Type II family with base $T^2$. In case 1-3, the $0$-end is PE, and in case 6-7 the $0$-end is ACH.} 
		\label{tab:model-cases}
	\end{table}

	\subsection{The case of Type II with base \texorpdfstring{$\Sigma_{\ttg}$}{Sigma g}}
	\label{subsec:model-solution-Type-II-sigma_g}

	We now consider the case where the base $(\Sigma_{\ttg},g_{\Sigma_{\ttg}})$ is a hyperbolic surface with constant Gauss curvature $-1$.
	For solutions of \eqref{eq:case-II} depending only on $\xi$, we have
	\[
	We^v=b+\frac a2\xi,
	\qquad
	e^v=b+a\xi-\frac b6\xi^3-\frac a{24}\xi^4-\xi^2.
	\]
	Set
	\[
	Q(\xi):=b+a\xi-\frac b6\xi^3-\frac a{24}\xi^4-\xi^2.
	\]
	Then
	\begin{equation}\label{eq:model_metric_cusp_cross_hyperbolic_surface}
		h_{\mathrm{mod}}
		=
		\frac1{\xi^2}\left(
		\frac{b+\frac a2\xi}{Q(\xi)} d\xi^2
		+
		\frac{Q(\xi)}{b+\frac a2\xi} \eta^2
		+
		\left(b+\frac a2\xi\right)g_{\Sigma_{\ttg}}
		\right),
		\qquad
		d\eta=\frac a2 d\vol_{\Sigma_{\ttg}}.
	\end{equation}

	Note that $Q$ cannot have a triple root. For $h_{\mathrm{mod}}$ to define a complete Riemannian metric with two ends, it is necessary that $Q$ have a double root at some $\xi_\ast\in\mathbb{R}$ satisfying $b+\frac{a}{2}\xi_\ast\neq 0$. Since $Q(\xi_\ast)=Q'(\xi_\ast)=0$, the identity
	\[
	Q(\xi)-\frac{\xi}{2}Q'(\xi)=\left(b+\frac{a}{2}\xi\right)\left(1+\frac{\xi^3}{12}\right)
	\]
	implies $\xi_\ast=-\sqrt[3]{12}$. Substituting this into $Q(\xi_\ast)=0$ yields
	$
	b+\frac{a}{2}\xi_\ast=\frac{\xi_\ast^2}{3}.
	$

	Conversely, if
	\[
	\xi_\ast=-\sqrt[3]{12}, \qquad b>0, \qquad b+\frac{a}{2}\xi_\ast=\frac{\xi_\ast^2}{3},
	\]
	then $h_{\mathrm{mod}}$ defines a complete Riemannian metric on $(\xi_\ast,0)$ with two ends. Indeed,

	\begin{equation}\label{eq:Q_expression_sigma_g_cusp}
		Q(\xi)=(\xi-\xi_\ast)^2\left(\frac{1}{3}-\frac{\xi_\ast^2}{18}\xi+\frac{a}{24}(\xi_\ast^2-\xi^2)\right)
	\end{equation}
	as $\xi\to\xi_\ast^+$,
	\[
	h_{\mathrm{mod}}
	\sim \frac{1}{3}\left(
	\frac{d\xi^2}{(\xi-\xi_\ast)^2}
	+
	\frac{9(\xi-\xi_\ast)^2}{\xi_\ast^4}\eta^2
	+
	g_{\Sigma_{\ttg}}\right).
	\]
	Therefore, the metric defined by \eqref{eq:model_metric_cusp_cross_hyperbolic_surface} on $(\xi_\ast,0)$ with $Q(\xi)$ defined by \eqref{eq:Q_expression_sigma_g_cusp} defines a complete Riemannian metric with two ends: one Poincar\'e--Einstein end and one $\Sigma_{\ttg}$-cusp end.

	\section{Solving the boundary value problems}
	\label{sec:solve-boundary-value-problem}

	In this section we establish existence, uniqueness, and decay estimates for the boundary value problems \eqref{eq:bvp1}--\eqref{eq:bvp4}. We consider the following classes of PDEs. We shall eventually make a change of variables in the (twisted-)Toda equation, which puts it into this form. Let $(\Sigma,g_{\Sigma})$ be a closed surface. Consider
	\begin{equation}\label{eq:u-PDE-t}
		\Delta_{\Sigma}u+\Psi(t)\,(e^{u})_{tt}+B(t)(e^{u})_{t}+2K(t)\,(e^u-1)=0
		\qquad\text{on }[0,\infty)\times \Sigma.
	\end{equation}
	We suppose
	\begin{equation}\label{eq:coeff-assumptions}
		\begin{aligned}
			&0<c^{-1}\le \Psi(t)\le c,\qquad K(t)\le 0,\qquad \Psi,B,K\in C^\infty([0,\infty)),\\
			&\sup_{t\ge 0}\bigl(|\partial_t^j\Psi(t)|+|\partial_t^jB(t)|+|\partial_t^jK(t)|\bigr)<\infty,
			\qquad \forall j\ge 0.
		\end{aligned}
	\end{equation}
	By integrating \eqref{eq:u-PDE-t} over $\Sigma$, $\fint_{\Sigma}e^{u(t,\cdot)} d\vol_{\Sigma}$ satisfies the ODE
	\[\Psi(t)\left(\fint_{\Sigma}e^{u(t,\cdot)}d\vol_{\Sigma}\right)_{tt}+B(t)\left(\fint_{\Sigma}e^{u(t,\cdot)}d\vol_{\Sigma}\right)_{t}+2K(t)\left(\fint_{\Sigma}e^{u(t,\cdot)}d\vol_{\Sigma}-1\right)=0.\]
	We study solutions such that $\fint_{\Sigma}e^{u(t,\cdot)} d\vol_{\Sigma}=1$. The main conclusion of this section is

	\begin{theorem}[Well-posedness]\label{thm:wp-stability}
		Given $\varphi\in C^{2,\alpha}(\Sigma)$ with $\fint_{\Sigma}e^{\varphi}\,d\vol_{\Sigma}=1$, there exists a unique solution $u\in C^{2}(\Sigma\times[0,\infty))$ of \eqref{eq:u-PDE-t} such that
		\[
		u|_{t=0}=\varphi,\qquad u\in L^\infty(\Sigma\times[0,\infty)),\qquad
		\fint_{\Sigma} e^{u(t,\cdot)}\,d\vol_{\Sigma}=1\ \ \forall\,t\ge 0.
		\]
		Moreover, it satisfies
		\[
		\|u\|_{L^\infty(\Sigma\times[0,\infty))}\le \|\varphi\|_{L^\infty(\Sigma)}, \qquad u=O(e^{-\delta t}) \text{ as } t\to\infty \text{ for some } \delta>0.
		\]
	\end{theorem}
	We will actually prove the following exponential stability result that is stronger than the uniqueness part and decay part of the above.
	\begin{theorem}[Exponential stability]\label{thm:stability-exp-decay}
		Let $u_i\ (i=1,2)$ be bounded solutions with boundary values $\varphi_i$ satisfying
		$\fint_{\Sigma}e^{\varphi_i} d\vol_{\Sigma}=1$ and $\|\varphi_i\|_{C^{2,\alpha}(\Sigma)}\le A$. Then there exist constants
		$C$ and $\delta>0$, depending on $\Sigma$, $g_{\Sigma}$, $A$, and coefficient bounds for $\Psi$, $B$, and $K$, such that for all $t\ge 0$ and $x\in \Sigma$,
		\[
		|u_1-u_2|(t,x)+|D(u_1-u_2)|(t,x)+|D^2(u_1-u_2)|(t,x)
		\le C \|\varphi_1-\varphi_2\|_{C^{2,\alpha}(\Sigma)}^{1/4} e^{-\delta t}.
		\]
	\end{theorem}

	Taking $u_2=0$ in Theorem \ref{thm:stability-exp-decay}, we get exponential decay of the solution, and taking $\varphi_1=\varphi_2$ we get the uniqueness part in Theorem \ref{thm:wp-stability}. 
    
    We prove the existence first.

	\begin{proposition}\label{prop:exhaustion-exterior-solution}
		Given $\varphi\in C^{2,\alpha}(\Sigma)$ with $\fint_{\Sigma}e^{\varphi} d\vol_{\Sigma}=1$, there exists
		$
		u\in C^{2}\big([0,\infty)\times \Sigma\big)\cap L^{\infty}\big([0,\infty)\times \Sigma\big)
		$
		solving \eqref{eq:u-PDE-t} with
		\[
		u|_{t=0}=\varphi,\qquad
		\fint_{\Sigma}e^{u(t,\cdot)} d\vol_{\Sigma}=1\quad \forall t\ge 0,\qquad
		\|u\|_{L^\infty(\Sigma\times[0,\infty))}\le \|\varphi\|_{L^\infty(\Sigma)}.
		\]
	\end{proposition}

	\begin{proof}
		Fix $k_{0}>0$ and consider the Dirichlet problem for \eqref{eq:u-PDE-t} on $[0,k_{0}]\times \Sigma$ with
		\begin{equation}\label{eq:dirichlet-data-finite-cylinder-sigma}
			u|_{t=0}=\varphi,\qquad u|_{t=k_{0}}=0.
		\end{equation}
		Integrating \eqref{eq:u-PDE-t} over $\Sigma$ and setting
		\[
		m(t):=\fint_{\Sigma}e^{u(t,\cdot)} d\vol_{\Sigma}
		\]
		we get the ODE
		\begin{equation}\label{eq:mass-ode-sigma}
			\Psi(t) m''(t)+B(t) m'(t)+2K (m(t)-1)=0.
		\end{equation}
		Since $m(0)=\fint_{\Sigma}e^{\varphi} d\vol_{\Sigma}=1$ and $m(k_{0})=\fint_{\Sigma}1 d\vol_{\Sigma}=1$, the maximum principle for the second--order
		linear ODE \eqref{eq:mass-ode-sigma} implies $m\equiv 1$ on $[0,k_{0}]$. In particular, for \eqref{eq:u-PDE-t} on $[0,k_0]\times \Sigma$ with the boundary values \eqref{eq:dirichlet-data-finite-cylinder-sigma}, its solution satisfies $m(t)=1$.

		We solve the boundary value problem on $[0,k_0]\times \Sigma$ by the continuity method.
		Let
		\[
		\begin{aligned}
			\mathcal{C}
			:=\Bigl\{ s\in[0,1]\ \Big|\ &
			\eqref{eq:u-PDE-t}\ \text{admits a solution }u\in C^{2,\alpha}([0,k_0]\times \Sigma)\\
			&\text{with }u|_{t=0}=\log\Bigl(\tfrac{e^{s\varphi}}{\fint_{\Sigma} e^{s\varphi}d\vol_{\Sigma}}\Bigr),\quad
			u|_{t=k_0}=0\Bigr\}.
		\end{aligned}
		\]
			Clearly, $0\in\mathcal{C}$.

		We first prove closedness of $\mathcal{C}$. The a priori bounds required for closedness of $\mathcal{C}$ follow from the maximum principle, De Giorgi--Nash--Moser theory, and Schauder estimates. Since $\Psi$ is uniformly positive and $\Psi,B$ have uniformly
		bounded derivatives, the constants in these estimates depend only on $\alpha$, $\|\varphi\|_{C^{2,\alpha}(\Sigma)}$,
		$c$, and finitely many uniform bounds for $\Psi,B$; in particular they can be chosen independent of $k_{0}$.

		For openness of $\mathcal{C}$, it suffices to show that the Dirichlet kernel of the formal adjoint of the linearized operator is trivial.
		Linearizing \eqref{eq:u-PDE-t} at a given function $u$ yields
		\[
		L_{u}\phi
		=\Delta_{\Sigma}\phi+\Psi(t) \partial_{tt}\bigl(e^{u}\phi\bigr)+B(t) \partial_{t}\bigl(e^{u}\phi\bigr)+2K e^{u}\phi,
		\]
		whose formal adjoint with respect to $dt\wedge d\vol_{\Sigma}$ is
		\[
		L_{u}^{*}f
		=\Delta_{\Sigma}f+e^{u}\Big((\Psi f)_{tt}-(B f)_{t}\Big)+2K e^{u}f.
		\]
		Let $f\in C^{2}([0,k_{0}]\times \Sigma)$ solve
		\begin{equation}\label{eq:adjoint-dirichlet-sigma}
			L_{u}^{*}f=0\quad\text{in }(0,k_{0})\times \Sigma,\qquad f|_{t=0}=f|_{t=k_{0}}=0.
		\end{equation}
		Choose $q=q(t)$ so that
		\[
		(\Psi e^{q})_{t}=Be^{q}\qquad\Longleftrightarrow\qquad q'=\frac{B-\Psi_t}{\Psi},
		\]
		and set $\tilde f:=e^{-q}f$ (so $f=e^{q}\tilde f$). Then
		\[
		e^{-q}L_u^{*}(e^{q}\tilde f)
		=\Delta_{\Sigma}\tilde f+e^{u}\bigl(\Psi \tilde f_{tt}+B \tilde f_{t}+2K \tilde f\bigr).
		\]

		Since $\Psi>0$ and $e^{u}$ is bounded above and below on $[0,k_{0}]\times \Sigma$, this is uniformly elliptic on the
		compact cylinder. Moreover, since $K\leq 0$, the zero--order coefficient $2K e^{u}\le 0$, so the maximum principle
		applies to \eqref{eq:adjoint-dirichlet-sigma} and yields $\tilde f\equiv 0$, hence $f\equiv 0$. Therefore
		$\ker_{\mathrm D}L_{u}^{*}=\{0\}$ and openness follows. Hence $\mathcal{C}=[0,1]$.

		Thus, for each $k_{0}>0$ there exists a solution
		$u_{k_{0}}\in C^{2,\alpha}([0,k_{0}]\times \Sigma)$ of \eqref{eq:u-PDE-t} with \eqref{eq:dirichlet-data-finite-cylinder-sigma}.
		By the $k_{0}$-independent estimates above, after passing to a subsequence we have
		$u_{k_{0}}\to u_\infty$ in $C^{2}$ on compact subsets of $[0,\infty)\times \Sigma$ as $k_{0}\to\infty$. The limit
		$u_\infty\in C^{2}([0,\infty)\times \Sigma)\cap L^\infty$ solves \eqref{eq:u-PDE-t}, satisfies $u_\infty|_{t=0}=\varphi$, and $\fint_{\Sigma}e^{u_\infty(t,\cdot)} d\vol_{\Sigma}=1$ passes to the limit for every $t\ge 0$.
		Finally, the $L^\infty$ estimate from the maximum principle on each finite cylinder $[0,k_{0}]\times \Sigma$ is uniform
		in $k_{0}$, hence it passes to the limit and yields
		\[\|u_\infty\|_{L^\infty(\Sigma\times[0,\infty))}\le \|\varphi\|_{L^\infty(\Sigma)}.\]
	\end{proof}

		Now we prove Theorem \ref{thm:stability-exp-decay} through a slice-wise $H^{-1}$ energy differential inequality. Proposition \ref{prop:exhaustion-exterior-solution} and Theorem \ref{thm:stability-exp-decay} together imply Theorem \ref{thm:wp-stability}.
	Let $I\subset\mathbb R$ be an interval and suppose $w:I\times \Sigma\to\mathbb R$ is a function where
	\begin{itemize}
		\item $w(t,\cdot)\in L^{2}(\Sigma)$ for each $t\in I$;
		\item $\int_{\Sigma}w(t,x) d\vol_{\Sigma}=0$.
	\end{itemize}
	For each $t$, let $\phi(t,\cdot)\in H^{2}(\Sigma)$ be the unique solution of
	\[
	-\Delta_{\Sigma}\phi(t,\cdot)=w(t,\cdot),\qquad \int_{\Sigma}\phi(t,x) d\vol_{\Sigma}=0,
	\]
	and define
	\[
	E_{w}(t):=\frac12\|w(t,\cdot)\|_{H^{-1}(\Sigma)}^{2}
	:=\frac12\int_{\Sigma}|\nabla\phi(t,x)|^{2} d\vol_{\Sigma}=\frac{1}{2}\int_{\Sigma}w(t,x)\phi(t,x)d\vol_{\Sigma}.
	\]
	Let $u_i\in C^{2}([0,\infty)\times \Sigma)\cap L^\infty([0,\infty)\times \Sigma)$ $(i=1,2)$ be solutions of
	\eqref{eq:u-PDE-t}
	with boundary values $u_i(0,\cdot)=\varphi_i$, $\|\varphi_i\|_{L^\infty(\Sigma)}\le A$, and $A':=\max_i\|u_i\|_{L^\infty([0,\infty)\times \Sigma)}<\infty$. Suppose $\fint_{ \Sigma} e^{u_i(t,\cdot)} d\vol_{\Sigma}=1$ for $t\geq0$.
	Define
	\[
	w(t,x):=e^{u_1(t,x)}-e^{u_2(t,x)}.
	\]
	We first prove a differential inequality.
	\begin{lemma}\label{lem:energy-inequality}
		There exist constants $\kappa_0,\kappa_1>0$, depending only on $A,A'$, $\inf\Psi, \sup\Psi$, and $\|B\|_{L^\infty}$, such that
		\begin{equation}\label{eq:abs-ODE-sigma-Kt}
			E_w''(t)\ge \kappa_0 E_w(t)-\kappa_1 |E_w'(t)|,\qquad\forall t\ge 0.
		\end{equation}
	\end{lemma}
	\begin{proof}
		Set $A'=\max_i\|u_i\|_{L^\infty([0,\infty)\times \Sigma)}<\infty$. Then $e^{-A'}\le e^{u_i}\le e^{A'}$.
		Note that $\int_{\Sigma} w(t,\cdot) d\vol_{\Sigma}=0$ for all $t\ge 0$. We have
		\begin{equation}\label{eq:Theta-def-sigma-Kt}
			u_1-u_2=\log(e^{u_1})-\log(e^{u_2})
			=\Theta w,
			\qquad
			\Theta(t,x):=\int_{0}^{1}\frac{1}{s e^{u_1(t,x)}+(1-s)e^{u_2(t,x)}} ds,
		\end{equation}
		hence
		\begin{equation}\label{eq:Theta-bds-sigma-Kt}
			e^{-A'}\le \Theta(t,x)\le e^{A'}\qquad\text{on }[0,\infty)\times \Sigma.
		\end{equation}
		Subtracting the equations for $u_1$ and $u_2$ gives
		\begin{equation}\label{eq:w-eq-sigma-Kt}
			\Delta_{\Sigma}(\Theta w)+\Psi(t) w_{tt}+B(t) w_t+2K(t) w=0.
		\end{equation}

		Let $\phi(t,\cdot)$ be the unique solution of $-\Delta_{\Sigma}\phi(t,\cdot)=w(t,\cdot)$ with $\int_{\Sigma}\phi(t,\cdot)d\vol_{\Sigma}=0$. Then
		\[
		E_w(t)=\frac12\int_{\Sigma}|\nabla\phi|^{2} d\vol_{\Sigma}=\frac12\int_{\Sigma}w \phi d\vol_{\Sigma},\qquad
		E_w'(t)=\int_{\Sigma} w_t \phi d\vol_{\Sigma},
		\]
		Integrating by parts and invoking \eqref{eq:w-eq-sigma-Kt} together with $K(t)\le 0$, we obtain
		\begin{align*}
			E_w''(t)
			&=\|w_t(t,\cdot)\|_{H^{-1}(\Sigma)}^{2}+\int_{\Sigma} w_{tt} \phi d\vol_{\Sigma}
			\;\ge\;\int_{\Sigma} w_{tt} \phi d\vol_{\Sigma} \\
			&=\frac{1}{\Psi(t)}\int_{\Sigma}\Bigl(-\Delta_{\Sigma}(\Theta w) \phi - B(t) w_t \phi-2K(t) w \phi\Bigr) d\vol_{\Sigma} \\
			&=\frac{1}{\Psi(t)}\int_{\Sigma}\Theta w^{2} d\vol_{\Sigma}-\frac{B(t)}{\Psi(t)} E_w'(t)-\frac{4K(t)}{\Psi(t)} E_w(t) \\
			&\ge \frac{e^{-A'}}{\Psi(t)}\int_{\Sigma} w^{2} d\vol_{\Sigma}-\frac{B(t)}{\Psi(t)} E_w'(t),
		\end{align*}
		Since $\int_{\Sigma}w d\vol_{\Sigma}=0$, we have
		\[
		\int_{\Sigma}w^{2} d\vol_{\Sigma} \ge \lambda_{1}(-\Delta_{\Sigma})\|w\|_{H^{-1}(\Sigma)}^{2}=2\lambda_{1}(-\Delta_{\Sigma})E_w(t),
		\]
		with $\lambda_{1}(-\Delta_{\Sigma})>0$ being the first non-zero eigenvalue of $-\Delta_{\Sigma}$. Hence there exist constants $\kappa_0,\kappa_1>0$
		(depending only on $A'$, $\inf\Psi, \sup\Psi$, $\|B\|_{L^\infty}$) such that
		\[
		E_w''(t)\ge \kappa_0 E_w(t)-\kappa_1 |E_w'(t)|
		\qquad\forall t\ge 0,
		\]
		which is \eqref{eq:abs-ODE-sigma-Kt}.
	\end{proof}

	Next we conclude the decay of the $H^{-1}$ energy.
	\begin{lemma}
		We have
		\[
		E_w(t)\le E_w(0)e^{-\kappa_2 t}\qquad\forall t\ge 0,
		\qquad
		\kappa_2:=\frac{\sqrt{\kappa_1^{2}+4\kappa_0}-\kappa_1}{2}.
		\]
	\end{lemma}
	\begin{proof}
		We claim $E_w'(t)\le 0$ for all $t\ge 0$. Otherwise, choose $t_0>0$ such that $E_w'(t_0)>0$ and $E_w(t_0)>0$, and set
		\[
		t_1:=\sup\bigl\{ t>t_0:\ E_w'(s)>0\ \text{for all }s\in[t_0,t) \bigr\}.
		\]
		On $[t_0,t_1)$, \eqref{eq:abs-ODE-sigma-Kt} implies $E_w''+\kappa_1 E_w'\ge \kappa_0 E_w$. Since $E_w$ is increasing on $[t_0,t_1)$, we have
		$E_w(t)\ge E_w(t_0)>0$ for all $t\in[t_0,t_1)$, hence
		$
		\bigl(e^{\kappa_1 t}E_w'(t)\bigr)'\ge \kappa_0 E_w(t_0)e^{\kappa_1 t}\ \text{on }[t_0,t_1).
		$
		Integrating over $[t_0,t]\subset [t_0,t_1)$ yields
		\[
		E_w'(t)\ge e^{-\kappa_1(t-t_0)}E_w'(t_0)+\frac{\kappa_0}{\kappa_1}E_w(t_0)\Bigl(1-e^{-\kappa_1(t-t_0)}\Bigr),
		\]
		so $E_w'(t)$ admits a positive lower bound on $[t_0,t_1)$, therefore $t_1=\infty$. Consequently $E_w(t)\to\infty$, contradicting the boundedness of $E_w$ (since $u_i$ are bounded).

		Thus $E_w'\le 0$, so $|E_w'|=-E_w'$ and
		\[
		E_w''-\kappa_1 E_w'-\kappa_0 E_w\ge 0.
		\]
		Let $r_\pm$ be the roots of $r^{2}-\kappa_1 r-\kappa_0=0$ and set $\kappa_2:=-r_->0$. Define $y:=E_w'-r_-E_w$. Then
		$y'-r_+y=E_w''-\kappa_1 E_w'-\kappa_0 E_w\ge 0$, so $e^{-r_+t}y(t)$ is nondecreasing. If $y(t_0)>0$ for some $t_0$, then
		$y(t)\ge y(t_0)e^{r_+(t-t_0)}$ and hence $E_w'(t)=y(t)+r_-E_w(t)$ becomes positive for $t\gg 1$, contradicting $E_w'\le 0$.
		Therefore $y\le 0$, i.e.\ $E_w'\le r_-E_w$. Integration gives
		\[
		E_w(t)\le E_w(0)e^{r_-t}=E_w(0)e^{-\kappa_2 t}.
		\]
	\end{proof}

	Finally we prove the exponential stability.
	\begin{lemma}
		There are constants $C, \kappa_3$ depending on $\|\varphi_i\|_{C^{2,\alpha}(\Sigma)}$ such that
		\begin{align}
			\sum_{j=0}^{10} \bigl|D^{j}(u_1-u_2)(t,x)\bigr|
			&\le
			Ce^{-\kappa_3t} \bigl\|{\varphi_1-\varphi_2}\bigr\|_{H^{-1}(\Sigma)}^{1/4},
			\qquad \forall t\ge 1,\ x\in \Sigma,
			\label{eq:uniform-stability-interior-general-PDE}
			\\
			\|u_1-u_2\|_{C^{2,\alpha}([0,2]\times \Sigma)}
			&\le
			C \bigl\|\varphi_1-\varphi_2\bigr\|_{C^{2,\alpha}(\Sigma)}^{1/4}.
			\label{eq:uniform-stability-boundary-general-PDE}
		\end{align}
	\end{lemma}
	\begin{proof}
		The De Giorgi--Nash--Moser and Schauder estimates, together with elliptic bootstrapping applied to \eqref{eq:u-PDE-t}, yield
		\[
		\sum_{k=0}^{10}\bigl|D^{k}u_i(t,x)\bigr|\le C \ \text{for all }(t,x)\in[1,\infty)\times \Sigma,
		\qquad
		\|u_i\|_{C^{2,\alpha}([0,2]\times \Sigma)}\le C.
		\]
		Here $C$ depends only on $A',\Psi,B,K$. By the following interpolation Lemma \ref{lem:interpolation}, we have
		\[
		\lim_{t\to\infty}\sup_{x\in \Sigma}|u_i(t,x)|=0,
		\]
		and the maximum principle then implies
		\[
		\sup_{[0,\infty)\times \Sigma} |u_i|\le \sup_{\Sigma} |\varphi_i|.
		\]
		In particular, we may take $A'=A$.
		Estimates \eqref{eq:uniform-stability-interior-general-PDE}--\eqref{eq:uniform-stability-boundary-general-PDE}
		follow by applying Schauder estimates to the linear equation for $w$ in \eqref{eq:w-eq-sigma-Kt}, combining this with
		Lemma~\ref{lem:interpolation}, and using the relation $u_1-u_2=\Theta w$.
	\end{proof}

	\begin{lemma}[Interpolation]\label{lem:interpolation}
		For $w\in H^{2}(\Sigma)$ such that $\int_{\Sigma} w\,d\vol_{\Sigma}=0$, we have
		\[\|w\|_{L^\infty(\Sigma)}\leq C\|w\|_{H^{\frac{5}{4}}(\Sigma)}\leq C\|w\|_{H^{-1}(\Sigma)}^{\frac14}\|w\|_{H^{2}(\Sigma)}^{\frac{3}{4}}\]
	\end{lemma}

	Finally we apply our study to the boundary value problems \eqref{eq:bvp1}--\eqref{eq:bvp4} and prove Theorem \ref{thm:bvp-existence-uniqueness} and \ref{thm:bvp-model-asymptotics}.

	\subsection{The boundary value problem \eqref{eq:bvp1}}
	\label{subsec:hyperbolic-cusp}

	We consider \eqref{eq:bvp1}. Integrating the Toda equation over $T^{2}$ at fixed $\xi$ gives
	\begin{equation}\label{eq:slice-affine-hyperbolic-cusp}
		\fint_{T^{2}}e^{v(\xi,\cdot)} d\vol_{T^2}=a\xi+b
		\quad\text{for some }a,b\in\mathbb{R}.
	\end{equation}
	Since the left-hand side is positive on $[0,\infty)$ and for \eqref{eq:bvp1} we only consider bounded solutions, we have $a=0$ and $b>0$. Plugging in $v|_{\xi=0}=\varphi$, we find $b=\fint_{T^2}e^{\varphi}\dvol_{T^2}$. Setting $u=v-\log b$, we have $\fint_{T^2}e^{u(\xi,\cdot)}\dvol_{T^2}\equiv 1$. By Theorems \ref{thm:wp-stability} and \ref{thm:stability-exp-decay}, we obtain existence, uniqueness, and exponential decay of the solution to \eqref{eq:bvp1}. It remains to show $W=1-\frac{1}{2}\xi v_{\xi}>0$ so that the solutions to \eqref{eq:bvp1} correspond to Einstein metrics.

	\begin{proposition}\label{prop:W-positive-BVP1}
		Let $v$ be a solution of \eqref{eq:bvp1}.
		Then
		\[
		W=1-\frac{1}{2}\xi v_{\xi}>0 \quad \text{on }[0,\infty)\times T^{2}.
		\]
	\end{proposition}
	\begin{proof}
		For any $s\in [0,1]$, let $v_s$ be the unique bounded solution to \eqref{eq:bvp1} with the boundary value
		\[
		v_s|_{\xi=0}=\log\Bigl(\tfrac{e^{s\varphi}}{\fint_{T^2} e^{s\varphi}\dvol_{T^2}}\Bigr),
		\]
		and set $W_s:=1-\frac{1}{2}\xi (v_s)_{\xi}$.
		Define
		\[
		\mathcal{C}:=\bigl\{s\in [0,1]\ \big|\ W_s>0 \ \text{for all }(\xi,x)\in [0,\infty)\times T^2\bigr\}.
		\]
		Clearly, $0\in \mathcal{C}$. Theorem \ref{thm:wp-stability} implies $\lim_{\xi\rightarrow\infty}W_s=1$, thus if $W_s>0$, then $W_s\geq\delta(s)>0$. The stability Theorem \ref{thm:stability-exp-decay} implies $\mathcal{C}$ is open.

		We show that $\mathcal{C}$ is closed.
		Let $s_i\to s_\infty$ with $s_i\in\mathcal{C}$.
		By our reduction \eqref{eq:case-I}, for each $i$, we choose a connection $1$-form $\eta_{s_i}$ determined by \eqref{eq:deta} such that
		\[
		h_{s_i}:=\xi^{-2}\Bigl(W_{s_i}\,d\xi^2+W_{s_i}^{-1}\eta_{s_i}^{\,2}+W_{s_i}e^{v_{s_i}}g_{T^2}\Bigr)
		\]
		is Einstein with $\Ric(h_{s_i})=-3h_{s_i}$, and let $\mathcal{K}_{s_i}$ be the Killing field. In particular,
		\[
		|\mathcal{K}_{s_i}|_{h_{s_i}}^{2}=\xi^{-2}W_{s_i}^{-1}.
		\]
		By the Bochner formula,
		\begin{equation}\label{eq:Bochner-formula}
			\frac{1}{2}\Delta_{h_{s_i}}|\mathcal{K}_{s_i}|_{h_{s_i}}^{2}
			=|\nabla^{h_{s_i}}\mathcal{K}_{s_i}|_{h_{s_i}}^{2}+3|\mathcal{K}_{s_i}|_{h_{s_i}}^{2}\geq 0.
		\end{equation}
		Hence $|\mathcal{K}_{s_i}|_{h_{s_i}}^{2}=\xi^{-2}W_{s_i}^{-1}$ satisfies the maximum principle. For any $[a,b]\subset (0,\infty)$,
		\[
		\sup_{[a,b]\times T^2}\,\xi^{-2}W_{s_i}^{-1}(\xi,x)
		=\max\Bigl\{a^{-2}\sup_{T^2}W_{s_i}^{-1}(a,x),\ b^{-2}\sup_{T^2}W_{s_i}^{-1}(b,x)\Bigr\}.
		\]
		Letting $b\rightarrow\infty$, we have
		\[
		\sup_{[a,\infty)\times T^2}\,\xi^{-2}W_{s_i}^{-1}(\xi,x)
		=a^{-2}\sup_{T^2}W_{s_i}^{-1}(a,x).
		\]
		By Theorem~\ref{thm:wp-stability}, we may choose $a_0$ such that for any
		$a\in (0,a_0)$ and all $s_i$ we have
		\[
		\frac{1}{2}\le W_{s_i}(a,\cdot)\le 2.
		\]
		Consequently,
		\[
		W_{s_i}\ge \frac{a^{2}}{2\xi^{2}}\quad \text{on }[a,\infty)\times T^2,
		\]
		and passing to the limit gives
		\[
		W_{s_\infty}\ge \frac{a^{2}}{2\xi^{2}}\quad \text{on }[a,\infty)\times T^2.
		\]
		Since $a\in (0,a_0)$ is arbitrary, we conclude that $W_{s_\infty}>0$ on $[0,\infty)\times T^2$.
		Hence $\mathcal{C}=[0,1]$ and $W=W_1>0$.
	\end{proof}

	\subsection{The boundary value problem \eqref{eq:bvp2}}
	\label{subsec:achc-asd}

	We now consider \eqref{eq:bvp2}. Since we consider solutions with logarithmic growth, by integrating the Toda equation over $T^2$ we have $\fint_{T^2}e^vd\vol_{T^2}=b+a\xi$ with $a>0$. Define
	\[
	u:=v-\log(b+a\xi).
	\]
	Then the Toda equation is equivalent to
	\begin{equation}\label{eq:u-PDE-asd-complex-cusp}
		\Delta_{T^2}u+\bigl((b+a\xi)e^u\bigr)_{\xi\xi}=0,
	\end{equation}
	and we have $\fint_{T^2}e^{u(\xi,\cdot)}d\vol_{T^2}=1$. We consider bounded solutions of \eqref{eq:u-PDE-asd-complex-cusp}.

	Introduce
	\[
	t:=\sqrt{\frac{b+a\xi}{b}}\in [1,\infty).
	\]
	Then
	\[
	\partial_\xi=\frac{a}{2bt}\partial_t.
	\]
	Therefore \eqref{eq:u-PDE-asd-complex-cusp} becomes
	\begin{equation}\label{eq:PDE-PsiB}
		\Delta_{T^{2}}u+\Psi(t) (e^{u})_{tt}+B(t) (e^{u})_{t}=0,
		\qquad
		\Psi(t)=\frac{a^{2}}{4b},\ \ B(t)=\frac{3a^{2}}{4bt}.
	\end{equation}
	The coefficients satisfy \eqref{eq:coeff-assumptions}. Hence Theorems \ref{thm:wp-stability} and \ref{thm:stability-exp-decay} apply, giving existence, uniqueness, and exponential decay. The positivity $W:=1-\frac12\xi v_\xi>0$ follows by the same argument as in Proposition \ref{prop:W-positive-BVP1}.
	For any solution $v=\log(b+a\xi)+u$,
	\[
	W=1-\tfrac12 \xi v_\xi
	=1-\frac{a\xi}{2(b+a\xi)}-\frac{1}{2}\xi u_{\xi}.
	\]
	Since
	\[
	W(0,x)\equiv 1,
	\qquad
	\lim_{\xi\to\infty} W(\xi,x)=\frac12 \ \ \text{uniformly in }x\in T^{2},
	\]
	we can argue as in Proposition \ref{prop:W-positive-BVP1} to conclude $W>0$. This shows that, for each fixed $a>0$ and boundary value $\varphi$, \eqref{eq:bvp2} has a unique solution.

	\subsection{The boundary value problem \eqref{eq:bvp3}}
	\label{subsec:achc-typeII}

	Consider the twisted Toda equation on $(\xi_\ast,0)\times T^{2}$,
	\begin{equation}\label{eq:v-PDE}
		\Delta_{T^{2}} v + (e^{v})_{\xi\xi}
		+\xi \frac{12e^v-6\xi (e^v)_{\xi}}{12+\xi^3}=0.
	\end{equation}
	Recall that we have the model solution $v_{\mathrm{mod}}$ from Section \ref{subsec:model-solution-Type-II} that corresponds to a Type II PE metric with an ACH cusp
	\[
	e^{v_{\mathrm{mod}}}
	=
	-\frac a{24}(\xi-\xi_\ast)^3(\xi+\xi_\ast),
		\qquad a>0, \quad \xi_\ast=-\sqrt[3]{12}.
	\]
	This is the model solution from \textbf{Case 2} in Table \ref{tab:model-cases} of Section \ref{subsec:model-solution-Type-II} with $2b^3=3a^3$.
	Set
	\begin{equation}\label{eq:u-def}
		u(\xi,x):=v(\xi,x)-v_{\mathrm{mod}}(\xi).
	\end{equation}
	By picking the correct $a$, we have $\fint_{T^{2}} e^{u(\xi,\cdot)} d\vol_{T^2}=1$.
	Then \eqref{eq:v-PDE} is equivalent to
	\begin{equation}\label{eq:u-PDE-xi}
		\Delta_{T^{2}}u
		+e^{v_{\mathrm{mod}}}(e^{u})_{\xi\xi}
		+\left(2(e^{v_{\mathrm{mod}}})_{\xi}
		-\frac{6\xi^{2}}{12+\xi^{3}} e^{v_{\mathrm{mod}}}\right)(e^{u})_{\xi}=0.
	\end{equation}

	Introduce
	\[
	t:=\Bigl(1-\frac{\xi}{\xi_\ast}\Bigr)^{-1/2}\in[1,\infty),
	\qquad
	\xi=\xi_\ast(1-t^{-2}).
	\]
	Then
	\[
	e^{v_{\mathrm{mod}}}(t)
	=
	-\frac{a\xi_\ast}{2} \frac{2t^{2}-1}{t^{8}}, \qquad W_{\mathrm{mod}}
	=
	\frac{t^{6}}{2t^{2}-1},
	\]
	and \eqref{eq:u-PDE-xi} becomes
	\begin{equation}\label{eq:u-PDE-t-typeII}
		\Delta_{T^{2}}u+\Psi(t)(e^{u})_{tt}+B(t)(e^{u})_{t}=0,
	\end{equation}
	where
	\[
	\Psi(t)=-\frac{a}{8\xi_\ast}\left(2-\frac1{t^{2}}\right),
	\qquad
	B(t)=
	\frac{a}{8\xi_\ast}
	\frac{30t^{6}-33t^{4}+9t^{2}-1}{t^{3}(3t^{4}-3t^{2}+1)}.
	\]
	Since $a>0$ and $\xi_\ast<0$, we have
	\[
	0<-\frac{a}{8\xi_\ast}\le \Psi(t)\le -\frac{a}{4\xi_\ast},
	\]
	and $B$ is bounded on $[1,\infty)$. Therefore the coefficients satisfy \eqref{eq:coeff-assumptions}, so Theorems \ref{thm:wp-stability} and \ref{thm:stability-exp-decay} apply directly. This shows \eqref{eq:bvp3} has a unique solution.

	Moreover, for $t\in [1,\infty),$
	\[
	W
	=
	W_{\mathrm{mod}}
	-\frac{\frac12\xi u_\xi}{1+\frac{\xi^3}{12}}
	=
	\frac{t^{6}}{2t^{2}-1}
	-\frac{t^{7}(t^{2}-1)}{4(3t^{4}-3t^{2}+1)} u_t\ge c_1t^4-c_2t^5e^{-\kappa_2 t}
	\]
	for some $c_1,c_2,\kappa_2>0$. In particular,
	$
	\lim_{t\to\infty}W(t)=+\infty.
	$
	The argument in Proposition~\ref{prop:W-positive-BVP1} then implies $W>0$ on $[1,\infty)\times T^{2}$.

	\subsection{The boundary value problem \eqref{eq:bvp4}}
	\label{subsec:cuspxsurface-typeII}

	Consider the twisted Toda equation on $(\xi_\ast,0)\times \Sigma_{\ttg}$,
	\begin{equation}\label{eq:v-PDE-sigma_g}
		\Delta_{\Sigma_{\ttg}} v + (e^{v})_{\xi\xi}
		+\xi \frac{12e^v-6\xi (e^v)_{\xi}}{12+\xi^3}=-2.
	\end{equation}

	Recall that we have model solutions $v_{\mathrm{mod}}$ from Section \ref{subsec:model-solution-Type-II-sigma_g} that correspond to a Type II PE metric with $\Sigma_{\ttg}$ cusp
	\[
	e^{v_{\mathrm{mod}}}
	=(\xi-\xi_\ast)^2\left(\frac{1}{3}-\frac{\xi_\ast^2}{18}\xi+\frac{a}{24}(\xi_\ast^2-\xi^2)\right).
	\]
	Here $a$ is determined by
	$
	\fint_{\Sigma_{\ttg}} e^\varphi \dvol_{\Sigma_{\ttg}}+\frac{a}{2}\xi_\ast=\frac{1}{3}\xi_\ast^2.
	$

	Define
	\begin{equation}\label{eq:u-def-sigma_g}
		u(\xi,x):=v(\xi,x)-v_{\mathrm{mod}}(\xi).
	\end{equation}
	It follows that $\fint_{\Sigma_{\ttg}} e^{u(\xi,\cdot)} d\vol_{\Sigma_{\ttg}}=1$ and $u\in L^\infty((\xi_\ast,0)\times \Sigma_{\ttg})$.
	Then \eqref{eq:v-PDE-sigma_g} is equivalent to
	\begin{equation}\label{eq:u-PDE-xi-sigma_g}
		\Delta_{\Sigma_{\ttg}}u
		+e^{v_{\mathrm{mod}}}(e^{u})_{\xi\xi}
		+\left(2(e^{v_{\mathrm{mod}}})_{\xi}
		-\frac{6\xi^{2}}{12+\xi^{3}} e^{v_{\mathrm{mod}}}\right)(e^{u})_{\xi}-2(e^u-1)=0.
	\end{equation}

	Introduce
	\[
	t:=-\log\Bigl(1-\frac{\xi}{\xi_\ast}\Bigr)\in[0,\infty),
	\qquad
	\xi=\xi_\ast(1-e^{-t}).
	\]

	Substituting into \eqref{eq:u-PDE-xi-sigma_g}, we obtain
	\begin{equation}\label{eq:u-PDE-t-typeII-hyp}
		\Delta_{\Sigma_{\ttg}}u+\Psi(t)(e^{u})_{tt}+B(t)(e^{u})_{t}-2(e^{u}-1)=0,
	\end{equation}
	where
	\[
	\Psi(t)
	=
	1+\left(\frac{a\xi_\ast^2}{12}-\frac{2}{3}\right)e^{-t}
	-\frac{a\xi_\ast^2}{24}e^{-2t},
	\]
	and
	\[
	B(t)
	=
	2\Psi_t(t)-3\Psi(t)
	+\frac{6(1-e^{-t})^2}{3-3e^{-t}+e^{-2t}}\Psi(t).
	\]

	Since $e^{v_{\mathrm{mod}}}>0$ on $(\xi_\ast,0)$, the function $\Psi$ is positive on $[0,\infty)$; moreover $\Psi(t)\to1$ as $t\to\infty$, so there exist constants $0<c\le C$ such that
	\[
	c\le \Psi(t)\le C
	\qquad\text{for all }t\in[0,\infty).
	\]
	Also, $B$ is bounded on $[0,\infty)$. Therefore \eqref{eq:u-PDE-t-typeII-hyp} is uniformly elliptic with bounded coefficients, hence Theorems \ref{thm:wp-stability} and \ref{thm:stability-exp-decay} are applicable.

	Moreover,
	\[
	W_{\mathrm{mod}}
	=
	\frac{\frac{\xi_\ast^2}{3}+\frac{a}{2}(\xi-\xi_\ast)}{e^{v_{\mathrm{mod}}}}
	=
	\frac{e^{2t}\left(\frac13-\frac{a}{2\xi_\ast}e^{-t}\right)}{\Psi(t)}.
	\]
	Hence
	\[
	W
	=
	W_{\mathrm{mod}}
	-\frac{\frac12\xi u_\xi}{1+\frac{\xi^3}{12}}
	=
	\frac{e^{2t}\left(\frac13-\frac{a}{2\xi_\ast}e^{-t}\right)}{\Psi(t)}
	-\frac{e^{t}(e^{t}-1)}{2(3-3e^{-t}+e^{-2t})}\,u_t
	=\frac13 e^{2t}+O(e^{(2-\kappa_2)t})
	\]
	for some $\kappa_2>0$. In particular,
	$
	\lim_{t\to\infty}W(t)=+\infty.
	$
	The argument in Proposition~\ref{prop:W-positive-BVP1} then implies $W>0$ on $[0,\infty)\times \Sigma_{\ttg}$.

	\section{Degenerations}
	\label{sec:degeneration}

	In this section we investigate various degenerations of our PE manifolds (with a cusp).

	\subsection{Limiting geometry of a family of Type I PE metrics}

	Recall our Type I model solutions in Section \ref{subsec:model-solutions-Type-I}, where the model solution is $e^{v_{\mathrm{mod}}}=b+a\xi$ and the Einstein metric is
	\[h_{\mathrm{mod}}
	=
	\frac1{\xi^{2}}\left(
	\frac{b+\frac a2\xi}{b+a\xi}\,d\xi^{2}
	+
	\frac{b+a\xi}{b+\frac a2\xi}\,\eta^{2}
	+
	\left(b+\frac a2\xi\right)g_{T^{2}}
	\right),
	\qquad
	d\eta=\frac a2\,d\vol_{T^2}.\]
		Suppose $a>0,b>0$. After choosing a suitable period $\mpp$, the model solution induces a PE metric with an ACH cusp on $(0,\infty)\times \NN_{\ell}$. We then let $b\to0^+$. By \eqref{eq:period-constraint}, in this process the period $\mpp=\frac{a}{2\ell}$ does not change. In this limit, the solution $v_{\mathrm{mod}}$ converges to $\log(a\xi)$. Geometrically, if we take a pointed Cheeger-Gromov limit at a suitable base point, for instance at $\xi=1$, then the PE metrics with an ACH cusp converge to a complex hyperbolic metric, namely the metric associated to $\log(a\xi)$.

	We show that this convergence to the complex hyperbolic metric also occurs for the PE manifolds with an ACH cusp arising from \eqref{eq:bvp2}. More precisely, we work in the following setting. Recall that, for each boundary value $\varphi$ and constant $a>0$, there is a unique solution $v$ to \eqref{eq:bvp2} with boundary value $\varphi$ satisfying
	\[\left|v-\log(\fint_{T^2} e^{\varphi}\dvol_{T^2}+a\xi)\right|=O(e^{-\delta\sqrt{\xi}}).\]
	As discussed at the end of Section \ref{subsec:the-boundary-value-problem}, solutions corresponding to different values of $a$ determine isometric Einstein manifolds, after shifting the boundary value by an appropriate constant. We may assume without loss of generality that $a=1$. Next, fix $\phi\in C^\infty(T^2)$ and, for $N\gg1 $, let $v_N$ denote the unique solution to \eqref{eq:bvp2} with $a=1$ and boundary data $\varphi=-N+\phi$. We prove the following.

	\begin{theorem}\label{thm:PE-to-ACH-degeneration-PDE}
		As $N\to\infty$, for any $D\Subset (0,\infty)\times T^2$, $v_N$ converges smoothly to $\log \xi$ on $D$.
	\end{theorem}
		Geometrically, for each solution $v_N$, after fixing a degree $\ell$ and choosing the correct period $\mpp$, we obtain a PE manifold with an ACH cusp $(M,h_N)$, where $M$ is diffeomorphic to $(0,\infty)\times \NN_\ell$. Let $p_N$ be a base point with $\xi(p_N)=1$. We then have
	\begin{theorem}\label{thm:PE-to-ACH-degeneration-geometry}
		As $N\to\infty$, $(M,h_N,p_N)$ converges smoothly in the pointed Cheeger-Gromov sense to the complex hyperbolic cusp metric.
	\end{theorem}

	It suffices to prove Theorem \ref{thm:PE-to-ACH-degeneration-PDE}.

	\begin{proof}[Proof of Theorem \ref{thm:PE-to-ACH-degeneration-PDE}]
		Denote $\bar\phi:=\log(\fint_{T^2}e^{\phi}\dvol_{T^2})$ and $C:=\sup_{T^2}|\phi-\bar\phi|$. Theorem \ref{thm:wp-stability} implies that
		\[\left|v_N-\log(e^{-N+\bar\phi}+\xi)\right|\leq C, \qquad \fint_{T^2} e^{v_N}\dvol_{T^2}=e^{-N+\bar\phi}+\xi. \]
		Thus, interior elliptic estimates imply that, after passing to a subsequence, $v_N$ converges smoothly on any $D\Subset (0,\infty)\times T^2$ to $v_\infty\in C^\infty((0,\infty)\times T^2)$, which satisfies \[(e^{v_\infty})_{\xi\xi}+\Delta_{T^2}v_\infty=0,\qquad |v_\infty-\log\xi|\leq C,\qquad \fint_{T^2}e^{v_\infty}\dvol_{T^2}=\xi.\]
		Lifting $v_\infty$ to $C^\infty((0,\infty)\times\mathbb{R}^2)$ and applying Lemma \ref{lem:Liouville_non_compact_easy}, we conclude $v_\infty=\log\xi$.
	\end{proof}

		An interesting phenomenon is that there is another non-trivial scale. By taking a pointed Cheeger-Gromov limit at suitably chosen base points, we obtain a non-trivial limit that is different from the complex hyperbolic space above.

	For each $N\gg1$, let $q_N\in M$ be a base point with $\xi(q_N)=e^{-N}$. We then analyze the pointed Cheeger--Gromov limit of the family $(M,h_N,q_N)$ as $N\to \infty$. Set $z:=\log\xi+N$.

	First we consider the model case. The metric becomes
	\[h_N=\frac{1+\frac{1}{2}e^z}{1+e^z}dz^2+\frac{1+e^z}{1+\frac{1}{2}e^z}\frac{1}{e^{-2N}e^{2z}}\eta^2+\frac{1+\frac{1}{2}e^z}{e^{-N}e^{2z}}g_{T^2}.\]
		Hence, taking the limit we obtain the following smooth limit Einstein metric on $\bR^4$, which is neither hyperbolic, nor complex hyperbolic.
	\begin{equation}\label{eq:blow-up-limit-Einstein-on-R4}
		h_{\infty}=\frac{1+\frac{1}{2}e^z}{1+e^z}dz^2+\frac{1+e^z}{1+\frac{1}{2}e^z}\frac{1}{e^{2z}}\left(dt+\frac{1}{2}xdy\right)^2+\frac{1+\frac{1}{2}e^z}{e^{2z}}(dx^2+dy^2).
	\end{equation}
	As $z\to\infty$, the metric $h_\infty$ is asymptotic to the complex hyperbolic space, whereas as $z\to-\infty$, it is asymptotic to the real hyperbolic space. Note that the metric $h_\infty$ was discovered in \cite{cortes2018quarter}.

	Next we investigate the general case. We identify the Riemann surface $T^2$ as $\mathbb{C}/\Lambda$, so the flat metric $g_{T^2}$ can be identified as $dx'^2+dy'^2$. The Toda equation then becomes
	\[(e^v)_{\xi\xi}+v_{x'x'}+v_{y'y'}=0.\]
	Set $\zeta=\xi e^N, x=x'e^{N/2},y=y' e^{N/2}$. For the solution $v_N$, we set
	\begin{equation}\label{eq:blow-up-R4-Einstein-metric}
		u_N(\zeta,x,y):=v_{N}(\xi,x',y')+N=v_N(e^{-N}\zeta,e^{-N/2}x, e^{-N/2}y)+N.
	\end{equation}
	Then $u_N$ satisfies
	\[(e^{u_N})_{\zeta\zeta}+(u_N)_{xx}+(u_N)_{yy}=0,\qquad u_N|_{\{\zeta=0\}}(x,y)=\phi(e^{-N/2}x,e^{-N/2}y).\]
	In terms of $u_N$, the Einstein metric is
	\[h_N=\zeta^{-2}\left(W(d\zeta^2+e^{u_N}(dx^2+dy^2)))+W^{-1}\eta^2\right),\qquad W=1-\frac{1}{2}\zeta (u_N)_{\zeta}.\]
	Denote $\bar{u}_N(\zeta)=\log(\fint_{T^2} e^{u_N(\zeta, x,y)} dxdy)=\log(e^{\bar\phi}+\zeta)$.
		Applying the maximum principle to the PDE for $w_N:=u_N-\bar{u}_N$, we have
	\[\sup |w_N|\leq \sup |\phi-\bar\phi|.\]
	Thus we can take a limit of $u_N$ to $u_\infty$, which satisfies the equation
	\[(e^{u_\infty})_{\zeta\zeta}+(u_\infty)_{xx}+(u_\infty)_{yy}=0,\quad u_\infty|_{\{\zeta=0\}}(x,y)=\phi(0,0), \]
		with $u_\infty-\log(e^{\bar\phi}+\zeta)\in L^\infty([0,\infty)\times\mathbb{R}^2).$ We have the following lemma.

	\begin{lemma}
		$u_\infty=\log(b+a\zeta)$ for some $a, b>0$.
	\end{lemma}

	\begin{proof}
		This is proved by Lemma \ref{lem:Liouville_non_compact} in the Appendix.
	\end{proof}

		Geometrically, if we consider $(M,h_N,q_N)$ with a base point $q_N$ where $\xi(q_N)=e^{-N},x'(q_N)=x_0',y'(q_N)=y_0'$, then the limit is again the Einstein metric on $\bR^4$ defined by
	\eqref{eq:blow-up-limit-Einstein-on-R4} (after coordinate change).

	By similar blow-up arguments, we provide a complete understanding of our Einstein metric during the process $N\to\infty$.

	\begin{theorem}\label{thm:pointed-limits}
		Let $\tilde p_N\in M$ be any sequence of base points. Then, after passing to a subsequence, the pointed spaces
		$
		(M,h_N,\tilde p_N)
		$
		converge in the pointed Gromov--Hausdorff sense, and every such pointed limit is one of the following:
		\begin{enumerate}
		\item $\bRH^4$, if $\xi(\tilde p_N)e^N\to 0$;
		
		\item the Einstein metric on $\mathbb{R}^4$ given in \eqref{eq:blow-up-limit-Einstein-on-R4}, if
		$\xi(\tilde p_N)e^N\to c\in(0,\infty)$;
		
		\item $\bCH^2$, if $\xi(\tilde p_N)e^N\to\infty$ and $\xi(\tilde p_N)\to0$;
		
		\item the complex hyperbolic cusp metric, if $\xi(\tilde p_N)\to c\in(0,\infty)$;
		
		\item $\mathbb{R}$, if $\xi(\tilde p_N)\to\infty$.
		\end{enumerate}
	\end{theorem}

	Note that in the third case, instead of using \eqref{eq:blow-up-R4-Einstein-metric}, we use $u_N(\zeta,x,y)=v_N(\lambda_N\zeta,\lambda_N^{1/2}x,\lambda_N^{1/2}y)-\log(\lambda_N)$ with $\lambda_N=\xi(\tilde{p}_N)$ and we apply Lemma \ref{lem:Liouville_non_compact_easy}.

	\subsection{Generalized Dehn filling}

	In this section we outline how to perform generalized Dehn filling on the
    PE manifold with an AH cusp constructed above. This gives a proof of
    Theorem \ref{thm:Dehn-filling}. We follow the approach of \cite{Biquard2008DGGA}. Suppose $(M,h)$ is an ASD PE manifold with an AH cusp that arises from our construction. Since the end $E_\infty$ is an AH cusp, we may choose a diffeomorphism identifying $E_\infty$ with $(0,\epsilon_0)\times T^3$ such that, for some small $\delta_0>0$,
	\begin{equation}\label{eq:asymptotic-to-cusp}
		h=\frac{ds^2}{s^2}+s^2g_{T^3}+O(s^{\delta_0})=h_{\bRH^4_{cusp}}+O(s^{\delta_0}).
	\end{equation}
	Here the cusp is at $s=0$, and the asymptotics are understood in the sense that $|\nabla_{h_{\bRH^4_{cusp}}}^k (h-h_{\bRH^4_{cusp}})|_{h_{\bRH^4_{cusp}}} = O(s^{\delta_0})$ for all $k \ge 0$. There is a toral black hole metric on $\bR^2\times\bR^2$
	\begin{equation}\label{eq:BH-metric}
		h_{BH}:=\frac{ds^2}{V(s)}+V(s)d\theta^2+s^2 g_{\bR^2},\quad V(s)=s^2-\frac{a}{s},
	\end{equation}
	where $s_+=a^{\frac{1}{3}}$, $s\in[s_+,\infty)$, and the period for $\theta$ is $\beta=\frac{4\pi}{3s_+}$. A different choice of $a$ gives the same metric up to isometry. The metric $h_{BH}$ is asymptotic to the hyperbolic metric $\frac{ds^2}{s^2}+s^2(d\theta^2+g_{\bR^2})$ as $s\to\infty$, with
	\[\left|h_{BH}-\left(\frac{ds^2}{s^2}+s^2(d\theta^2+g_{\bR^2})\right)\right|<C\frac{a}{s^3}\]
	for a constant $C>0$, measured under the hyperbolic metric.

	Next we describe the approximation metric. We would like to patch \eqref{eq:asymptotic-to-cusp} and \eqref{eq:BH-metric} near $s=e^{-R}$ for $R\gg1$. For this, we pick a primitive closed geodesic $\sigma$ on the flat $T^3$ with length $l\gg1$, such that
	\[\beta^2V(e^{-R})=l^2.\]
	This is equivalent to
	\[\frac{16\pi^2(1-e^{3R}a)}{9a^{\frac{2}{3}}e^{2R}}=l^2.\]
	So once we fix $l,R$, we can determine $a$. There is a unique lattice $\Lambda\subset \bS^1\times\mathbb{R}^2$ such that
	\begin{itemize}
		\item there is an isometry that identifies the quotient of the slice $s=e^{-R}$ in \eqref{eq:BH-metric} by $\Lambda$, which is a flat torus, to the $T^3$ at $s=e^{-R}$ of the hyperbolic cusp metric $h_{\bRH^4_{cusp}}=\frac{ds^2}{s^2}+s^2g_{T^3}$;
		\item the isometry identifies the $\theta$-geodesic in the quotient torus to $\sigma$.
	\end{itemize}

	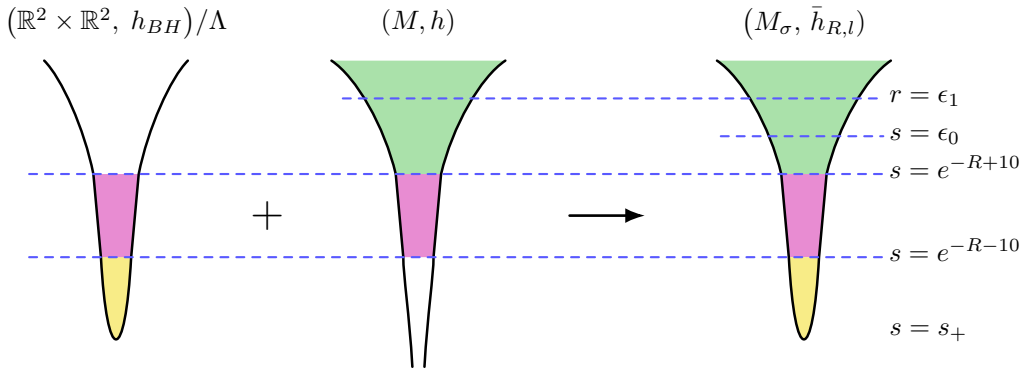
\begin{figure}[H]
		\begin{tikzpicture}[x=1cm,y=1cm,line join=round,line cap=round,>=Latex]

			\definecolor{myyellow}{RGB}{248,236,135}
			\definecolor{mypink}{RGB}{232,140,210}
			\definecolor{mygreen}{RGB}{170,225,170}
			\colorlet{guideblue}{blue!65}

			\def\yup{0.55}
			\def\ylow{-0.55}

			\begin{scope}[shift={(0,0)}]

				\path[fill=myyellow]
				(-0.2,\ylow)
				.. controls (-0.12,-2) and (0.12,-2).. (0.2,\ylow)
				-- cycle;

				\path[fill=mypink]
				(-0.30,\yup)
				.. controls (-0.27,0.2) and (-0.23,-0.2).. (-0.2,\ylow)
				-- (0.2,\ylow)
				.. controls (0.23,-0.2) and (0.27,0.2).. (0.30,\yup)
				-- cycle;

				\draw[line width=0.95pt]
				(-0.95,2.05)
				.. controls (-0.72,1.82) and (-0.42,1.15).. (-0.30,\yup)
				.. controls (-0.27,0.2) and (-0.23,-0.2).. (-0.2,\ylow)
				.. controls (-0.12,-2) and (0.12,-2).. (0.2,\ylow)
				.. controls (0.23,-0.2) and (0.27,0.2).. (0.30,\yup)
				.. controls (0.42,1.15) and (0.72,1.82).. (0.95,2.05);

				\node[font=\small] at (0,2.55) {$\bigl(\mathbb{R}^2\times \mathbb{R}^2,\; h_{{BH}}\bigr)/\Lambda$};

			\end{scope}
			\node[font=\bfseries\LARGE] at (2,0) {$+$};
			\begin{scope}[shift={(4,0)}]

				\path[fill=mygreen]
				(-1.15,2.05)
				.. controls (-0.82,1.83) and (-0.45,1.18).. (-0.30,\yup)
				-- (0.30,\yup)
				.. controls (0.45,1.18) and (0.82,1.83).. (1.15,2.05)
				-- cycle;

				\path[fill=mypink]
				(-0.30,\yup)
				.. controls (-0.27,0.2) and (-0.23,-0.2).. (-0.2,\ylow)
				-- (0.2,\ylow)
				.. controls (0.23,-0.2) and (0.27,0.2).. (0.30,\yup)
				-- cycle;

				\draw[line width=0.95pt]
				(-1.15,2.05)
				.. controls (-0.82,1.83) and (-0.45,1.18).. (-0.30,\yup)
				.. controls (-0.27,0.2) and (-0.23,-0.2).. (-0.2,\ylow)
				.. controls (-0.20,-0.8) and (-0.12,-1.3).. (-0.08,-2)

				(0.08,-2)
				.. controls (0.12,-1.3) and (0.20,-0.8).. (0.2,\ylow)
				.. controls (0.23,-0.2) and (0.27,0.2).. (0.30,\yup)
				.. controls (0.45,1.18) and (0.82,1.83).. (1.15,2.05);

				\node[font=\small] at (0,2.55) {$(M,h)$};

			\end{scope}

			\draw[->,line width=1pt] (6,0) -- (7,0);

			\begin{scope}[shift={(9.10,0)}]

				\path[fill=mygreen]
				(-1.15,2.05)
				.. controls (-0.82,1.83) and (-0.45,1.18).. (-0.30,\yup)
				-- (0.30,\yup)
				.. controls (0.45,1.18) and (0.82,1.83).. (1.15,2.05)
				-- cycle;

				\path[fill=mypink]
				(-0.30,\yup)
				.. controls (-0.27,0.2) and (-0.23,-0.2).. (-0.2,\ylow)
				-- (0.2,\ylow)
				.. controls (0.23,-0.2) and (0.27,0.2).. (0.30,\yup)
				-- cycle;

				\path[fill=myyellow]
				(-0.2,\ylow)
				.. controls (-0.12,-2) and (0.12,-2).. (0.2,\ylow)
				-- cycle;

				\draw[line width=0.95pt]
				(-1.15,2.05)
				.. controls (-0.82,1.83) and (-0.45,1.18).. (-0.30,\yup)
				.. controls (-0.27,0.2) and (-0.23,-0.2).. (-0.2,\ylow)
				.. controls (-0.12,-2) and (0.12,-2).. (0.2,\ylow)
				.. controls (0.23,-0.2) and (0.27,0.2).. (0.30,\yup)
				.. controls (0.45,1.18) and (0.82,1.83).. (1.15,2.05);

				\node[font=\small] at (0,2.55) {$\bigl(M_{\sigma},\, \bar h_{R,l})$};

				\node[font=\small] at (1.6,\yup+1) {$r=\epsilon_1$};
				\node[font=\small] at (1.6,\yup+0.5) {$s=\epsilon_0$};
				\node[font=\small] at (2,\yup+0.1) {$s=e^{-R+10}$};
				\node[font=\small] at (2,\ylow+0.1) {$s=e^{-R-10}$};
				\node[font=\small] at (1.65,\ylow-1) {$s=s_+$};

			\end{scope}
			\draw[guideblue,dashed,thick] (-1.15,\yup) -- (10.15,\yup);
			\draw[guideblue,dashed,thick] (-1.15,\ylow) -- (10.15,\ylow);
			\draw[guideblue,dashed,thick] (3,\yup + 1) -- (10.15,\yup + 1);
			\draw[guideblue,dashed,thick] (8,\yup + 0.5) -- (10.15,\yup + 0.5);

		\end{tikzpicture}
		\caption{The gluing construction}
		\label{fig:gluing-dehn-filling}
	\end{figure}

	The quotient of $h_{BH}$ by $\Lambda$ yields an Einstein metric on $\bR^2\times T^2$.
	Denote $M\setminus \{0<s<e^{-R}\}$ by $M_R$ and $\bR^2\times T^2 \setminus \{s\ge e^{-R}\}$ by $N_R$.
	Via the isometry between the boundary tori, we identify the slice $s=e^{-R}$ of $h_{\bRH^4_{cusp}}$ with the slice $s=e^{-R}$ of the quotient toral black hole metric, thereby obtaining the manifold $M_\sigma = M_R \amalg N_R$. The function $s$ on the end $(0,\epsilon_0)\times T^3$ patches continuously with the one on $D^2\times T^2$ at $s=e^{-R}$, hence we may define the function $s$ on $M_\sigma$, via extending it as $s\equiv\epsilon_0$ outside of the end $(0,\epsilon_0)\times T^3$ in $M$. We define the approximate metric $\bar h_{R,l}$ on $M_\sigma$
	\begin{itemize}
		\item as $h$ on the part $s\geq e^{-R+10}$;
		\item as the quotient of $h_{BH}$ on the part $s\leq e^{-R-10}$;
		\item we interpolate them smoothly via a cut-off function $\gamma = \chi(\log(s)+R)$ on $e^{-R-10} < s < e^{-R+10}$, where $\chi \in C^\infty(\mathbb{R})$, $0\leq \chi\leq 1$, $\chi \equiv 0$ on $[10, \infty)$, and $\chi \equiv 1$ on $(-\infty, -10]$.
	\end{itemize}
	So $\bar h_{R,l}=(1-\chi) h+\chi h_{BH}$.
	For the PE end $E_0$ of $M$, we can find a geodesic defining function $r$ and identify the Einstein metric $h$ with $\frac{1}{r^2}(dr^2+g_{r})$ on $(0,\epsilon_1/2)_r\times T^3$. The function $r$ can also be extended to a smooth function on $M_{\sigma}$ such that $r\equiv \epsilon_1$ on the region $\{s<\epsilon_0\}$.

	Next we define weighted H\"older norms for symmetric $2$-tensors on
$M_\sigma$. Fix weights
\[
0<\delta<\frac32,\qquad \frac32<\tau<3.
\]
Set
\[
\|u\|_{C^{m,\alpha}_{\delta,\tau}(M_\sigma)}
:=
\left\|
r^{-\tau}\left(\frac{s}{s_+}\right)^{\delta}u
\right\|_{C^{m,\alpha}(M_\sigma)}.
\]
	Here $m\geq 2, \ 0<\alpha<1$, and the $C^{m,\alpha}$-norm of a symmetric $2$-tensor $k$ is defined by
	\[
	\|k\|_{C^{m,\alpha}(M_\sigma)}
	:=
	\sup_{p\in M_\sigma}
	\bigl\|(\exp_p)^*k\bigr\|_{C^{m,\alpha}\left(B_{r_0}(0)\subset T_pM_\sigma\right)},
	\qquad
	r_0:=\frac12 \inf_{p\in M_\sigma}\operatorname{conj}_{\bar h_{R,l}}(p),
	\]
	where $\exp_p:T_pM_\sigma\to M_\sigma$ is the exponential map with respect to $\bar h_{R,l}$. Note that the sectional curvatures of $(M_\sigma, \bar{h}_{R,l})$ are uniformly bounded; thus, the conjugacy radius of $(M_\sigma, \bar{h}_{R,l})$ is bounded uniformly from below, independently of $\sigma$, $R$, and $l$.

	Recall that $s\equiv\epsilon_0$ outside a compact region of $M_\sigma$ while $r\equiv \epsilon_1$ inside a compact region of $M_\sigma$. For a background metric $\bar g$ and a nearby metric $g$,
	consider the Einstein operator in the Bianchi gauge
	\[\Phi_{\bar g}(g)=\Ric_{g}+3g+\divv_{g}^*\left(\divv_{\bar g}(g)+\frac{1}{2}d\mathrm{tr}_{\bar g}(g)\right),\]
	whose linearization at $\bar g$ is
	\[(d\Phi_{\bar g})_{\bar g}(k)=\frac{1}{2}\Delta_{L}k+3k.\]
	Here $\Delta_L$ is the Lichnerowicz Laplace of $\bar g$, given by 
    $(\Delta_L k)_{ij}=(\nabla^*\nabla k)_{ij}+\Ric_i{}^p k_{pj}+\Ric_j{}^p k_{ip}
    -2R_{ipqj}k^{pq}
    $. In what follows, we denote $L_{\bar g}:=(d\Phi_{\bar g})_{\bar g}$ to denote the linearization at $\bar g$. Note that our ASD PE
manifold with an AH cusp $(M,h)$ is \emph{non-degenerate}
\cite{BiquardRollin2009Wormholes}, in the sense that the kernel of
$
L_h:=(d\Phi_h)_h
$
is trivial in $L^2$.

    We prove the following uniform estimate of $L_{\bar h_{R,l}}$ for all sufficiently large $R$ and $l$.

	\begin{lemma}\label{lem:invertibility-of-linearization}
    There exist constants $A,C>0$ such that, for all $R,l>A$, if $u$ and $f$
    are symmetric $2$-tensors satisfying $L_{\bar h_{R,l}}u=f$, then
    \[
    \|u\|_{C^{m,\alpha}_{\delta,\tau}}
    \leq C\|f\|_{C^{m-2,\alpha}_{\delta,\tau}} .
    \]
    \end{lemma}
	\begin{proof}
		We prove this by contradiction. If not, then we have $R_j,l_j\to\infty$, symmetric 2 tensors $u_j$ with $L_ju_j=f_j$ where $\|u_j\|_{C^{m,\alpha}_{\delta,\tau}}=1,\|f_j\|_{C^{m-2,\alpha}_{\delta,\tau}}\to0$, $L_j:=L_{\bar h_{R_j,l_j}}$. Elliptic estimates imply that $\|u_j\|_{C^{0}_{\delta,\tau}}>\epsilon_2$ for some $\epsilon_2>0$, so we can find points $x_j\in M_{\sigma_j}$, with
		\begin{equation}\label{eq:positive-lower-bound-in-contradiction-argument}
			r(x_j)^{-\tau}\left(\frac{s(x_j)}{s_+}\right )^\delta |u_j(x_j)|>\epsilon_2.
		\end{equation}
        Here, for simplicity, we use the same notation $r$, $s$, and $s_+$ throughout,
        although they depend on $R_j$ and $l_j$.
		By passing to a subsequence, there are four possibilities.
		\begin{enumerate}
			\item $x_j$ diverges to the PE end, i.e. $r(x_j)\to0$;
			\item $x_j$ converges to compact regions of $M$, i.e. $r(x_j)$ and $s(x_j)$ remain uniformly bounded from below;
			\item $x_j$ converges to the hyperbolic region between $M$ and the black hole metric, i.e. $s(x_j)\to0$ and $s(x_j)/s_+\to\infty$.
			\item $x_j$ converges to compact regions of the black hole metric, i.e. $s(x_j)/s_+$ remains uniformly bounded from above;
		\end{enumerate}
		We first show that (2)-(4) cannot happen.

	If (2) happens, normalize by
    $\bar u_j:=s_+^{-\delta}u_j$ and $\bar f_j:=s_+^{-\delta}f_j$.
    After passing to a subsequence, $\bar u_j$ converges on compact subsets of
    $(M,h)$ to a nonzero tensor $\bar u_\infty$ satisfying
    \[
    L_h\bar u_\infty=0,
    \qquad
    \|r^{-\tau}s^\delta\bar u_\infty\|_{C^{m,\alpha}}<\infty .
    \]
    For $0<\delta<\frac32$, this bound implies that $\bar u_\infty$ is $L^2$
    along the AH cusp. For $\frac32<\tau<3$, it also implies that
    $\bar u_\infty$ is $L^2$ along the PE end. Hence
    $\bar u_\infty$ lies in the $L^2$ kernel of $L_h$ on $(M,h)$,
    contradicting the non-degeneracy of $(M,h)$.

		If (3) happens, then we write $w=\frac{s}{s(x_j)}$. Set $\bar u_j=(\frac{s(x_j)}{s_+})^\delta u_j$ and $\bar f_j=(\frac{s(x_j)}{s_+})^\delta f_j$, then $(\frac{s}{s_+})^\delta u_j=(\frac{s(x_j)}{s_+})^\delta w^\delta u_j$. Around $x_j$, the metric $\bar h_{R_j,l_j}$ is collapsing. But after passing to the local universal cover it is close to
		\[\frac{dw^2}{w^2}+w^2s(x_j)^2g_{\bR^3},\]
		which can be seen explicitly by substituting $w=s/s(x_j)$ into \eqref{eq:asymptotic-to-cusp} and \eqref{eq:BH-metric}. In particular, it is converging to the hyperbolic metric, and the tensor $\bar u_j$ converges to a non-trivial $\bR^3$-invariant limit $\bar u_\infty$ on the hyperbolic space $\bRH^4$ that satisfies $L_{h_{\bRH^4}}\bar u_\infty=0$ and $\|w^\delta\bar u_\infty\|_{C^{m,\alpha}}<\infty$. This is again impossible by Lemma 3.9.21 of \cite{Biquard2008DGGA}, where we need $0<\delta<3$.

		If (4) happens, then one can rewrite the black hole metric by setting $w=\frac{s}{s_+}$. Then the black hole metric becomes
		\[h_{BH}=\frac{dw^2}{w^2-1/w}+s_+^2(w^2-1/w)d\theta^2+s_+^2w^2g_{\bR^2}.\]
		The approximation metric $\bar h_{R_j,l_j}$ is collapsing around $x_j$, but after locally passing to the universal cover it is converging to $h_{BH}$. The value $w(x_j)$ is bounded from above and below. Hence passing to the limit of the local universal cover, we obtain a limit function $\bar u_\infty$ on $(\bR^2\times\bR^2,h_{BH})$ that satisfies $L_{h_{BH}}\bar u_\infty=0$ and $\|w^\delta \bar u_\infty\|_{C^{m,\alpha}}<\infty$. Then there is a contradiction as proved in Lemma 3.9.17 of \cite{Biquard2008DGGA}, where we need $0<\delta<3$.

		Finally, we show (1) cannot happen as well. Since (2)-(4) cannot happen, we know on compact regions of $M$ (that means, the set where $r,s$ are uniformly bounded from below), $|r^{-\tau}(\frac{s}{s_+})^\delta u_j|\to0$ uniformly. Note that on the compact regions of $M$ and the PE end, $s$ is a constant. Set $\bar u_j=(\frac{s}{s_+})^\delta u_j,\bar f_j=(\frac{s}{s_+})^\delta f_j$ and choose a cut-off function $\chi$ on $M$ that is equal to one outside a neighborhood of $E_\infty$ and is equal to zero outside a neighborhood of $M\setminus E_\infty$. The cut-off function therefore trivially extends to $M_\sigma$. Since (2)-(4) cannot happen and $\|\bar f_j\|_{C^{m-2,\alpha}_{\delta,\tau}}\to0$, we have
		\[\|r^{-\tau} \chi \bar u_j\|_{C^{m,\alpha}(M)}>1/2\]
		and
		\[\|r^{-\tau} L_j(\chi \bar u_j)\|_{C^{m-2,\alpha}(M)}<\epsilon\]
		for a small $\epsilon$ when $j\gg1$. However, since $(M,h)$ is non-degenerate, we have the inequality
		\[\|\chi \bar u_j\|_{C^{m,\alpha}_{\tau}(M)}\leq C\|L_h(\chi\bar u_j)\|_{C^{m-2,\alpha}_{\tau}(M)}\]
		where $C^{m,\alpha}_\tau$ is the weighted norm $\|u\|_{C^{m,\alpha}_\tau}=\|r^{-\tau} u\|_{C^{m,\alpha}}$ on $(M,h)$, and we need $0<\tau<3$. This leads to a contradiction as $L_h(\chi\bar u_j)=L_j(\chi\bar u_j)$.
	\end{proof}

It remains to estimate the defect \(\Phi_{\bar h_{R,l}}(\bar h_{R,l})\). Since
both \(h\) and \(h_{BH}\) are Einstein, this defect is supported in the transition
region \(A_R=\{e^{-R-10}<s<e^{-R+10}\}\). On \(A_R\), relative to
\(h_c=s^{-2}ds^2+s^2g_{T^3}\), we have
\(h=h_c+O(s^{\delta_0})\) and \(h_{BH}=h_c+O(as^{-3})\), with the same bounds
for derivatives. Since \(s\sim e^{-R}\) on \(A_R\), and the
cutoff \(\chi(\log s+R)\) has uniformly bounded derivatives, it follows that
\[
\|\Phi_{\bar h_{R,l}}(\bar h_{R,l})\|_{C^{m-2,\alpha}_{\delta,\tau}}
\le
C\left(\frac{e^{-R}}{s_+}\right)^\delta
\left(e^{-\delta_0R}+ae^{3R}\right).
\]
Here \(r\equiv \epsilon_1\) on \(A_R\), so the factor \(r^{-\tau}\) is harmless.
Using the matching relation \(s_+e^R\sim l^{-1}\), equivalently
\(ae^{3R}=(s_+e^R)^3\sim l^{-3}\), we obtain
\[
\|\Phi_{\bar h_{R,l}}(\bar h_{R,l})\|_{C^{m-2,\alpha}_{\delta,\tau}}
\le
C\bigl(l^\delta e^{-\delta_0R}+l^{\delta-3}\bigr).
\]
Thus, for $0<\delta<\frac32$ and for
$R,l\to\infty$ chosen so that $l^\delta e^{-\delta_0R}\to0$, the defect tends
to zero. For $0<\tau<3$, the operator
$
L_{\bar h_{R,l}}:C_{\delta,\tau}^{m,\alpha}\rightarrow C_{\delta,\tau}^{m-2,\alpha}
$
is bounded and Fredholm of index zero; by
Lemma~\ref{lem:invertibility-of-linearization}, it is invertible, and
$\|L_{\bar h_{R,l}}^{-1}\|$ is uniformly bounded independently of $R,l$. Moreover, the
nonlinear remainder satisfies the corresponding uniform quadratic estimate.
The inverse function theorem therefore gives a unique small perturbation
$u_{R,l}\in C_{\delta,\tau}^{m,\alpha}$ such that
$h_{R,l}:=\bar h_{R,l}+u_{R,l}$ solves
$\Phi_{\bar h_{R,l}}(h_{R,l})=0$ on $M_\sigma$. Since $\tau>0$, the Bianchi
one-form decays at the PE end; applying the standard Bianchi argument to the
gauged equation gives
$
\Ric_{h_{R,l}}+3h_{R,l}=0.
$
Thus $h_{R,l}$ is a Poincar\'e--Einstein metric on $M_\sigma$ with the same
conformal infinity as $(M,h)$.

\subsection{Further Discussion}

The examples constructed in Theorem \ref{thm:PE-with-AH-cusp} show that Proposition 5.7 of \cite{anderson2003boundary} fails. Their conformal infinities have negative Yamabe constants, except in the flat $T^3$ case. Recall that for a
closed $n$-manifold $N$, $n\geq 3$, with conformal class $[\gamma]$, the
Yamabe constant is
$$
Y(N,[\gamma])
:=
\inf_{\tilde\gamma\in[\gamma]}
\frac{\int_N R_{\tilde\gamma}\,d\vol_{\tilde\gamma}}
{\operatorname{Vol}(N,\tilde\gamma)^{\frac{n-2}{n}}}.
$$
We mention that Proposition 5.7 of \cite{anderson2003boundary} is true if we assume further that the conformal infinity of a PE end has non-negative Yamabe constant by the result of Witten-Yau \cite{WittenYau1999} and Cai-Galloway \cite{Cai1999BoundariesOZ}. For completeness we recall their argument.

\begin{proposition}\label{prop:Cai-Galloway-rigidity}
Let $(M^{n+1},h)$ be a connected complete Einstein manifold with
$\operatorname{Ric}_h=-nh$. Assume that $M$ has at least a Poincar\'e--Einstein end
whose conformal infinity $(N,[\gamma])$ satisfies $Y(N,[\gamma])\geq 0$.
Then one of the following alternatives holds:
\begin{enumerate}
\item $Y(N,[\gamma])>0$, and $M$ has only one end.
\item $Y(N,[\gamma])=0$, and $M$ has only one end.
\item $Y(N,[\gamma])=0$, and $(M,h)$ is isometric to
\begin{equation}\label{eq:warped-product-hyperbolic-metric}
        \mathbb R\times N,\qquad h=ds^2+e^{-2s}g_N, \qquad \operatorname{Ric}_{g_N}=0.
\end{equation}
\end{enumerate}
\end{proposition}

\begin{proof}
Choose a Yamabe representative $\gamma$ of the conformal infinity
of the given PE end, with $R_{\gamma}\equiv Y(N,[\gamma])\geq 0$. Let $r$ be the
associated geodesic defining function. Near the PE end,
$$
 h=r^{-2}(dr^2+g_r),\qquad g_r|_{r=0}=\gamma .
$$
For $N_r=\{r=\mathrm{const}\}$, the mean curvature satisfies
$$
 H_r=n+\frac{1}{2(n-1)}R_{\gamma}r^2+o(r^2).
$$

If $Y(N,[\gamma])>0$, then $H_r>n$ for all sufficiently small $r$. Hence, for some
$r_0,\delta>0$, the compact hypersurface $N_{r_0}$ has mean curvature
$H\geq n(1+\delta)$. Since $\operatorname{Ric}_h\geq -nh$, Riccati
comparison gives
$$
    d(x,N_{r_0})\leq \coth^{-1}(1+\delta)
$$
for every $x$ outside the PE collar $\{0<r<r_0\}$. Thus the complement of
the PE end is compact, and $M$ has no other end.

It remains to consider $Y(N,[\gamma])=0$. Let $r_k\to 0$ and set $\Sigma_k=N_{r_k}$.
Fix $o\in M$ and write $d_k=d(o,\Sigma_k)$. Since $r$ is geodesic,
$d_k=-\log r_k+O(1)$, hence $e^{2d_k}=O(r_k^{-2})$. If
$h_k=\min\{\min_{\Sigma_k}H_k,n\}$, then the expansion gives
$$
    (n-h_k)e^{2d_k}\to 0 .
$$
Theorem 3 of \cite{Cai1999BoundariesOZ}, together with
$\operatorname{Ric}_h=-nh$, implies that either (2) or (3) holds.
\end{proof}

For completeness, we also verify non-local conformal flatness and determine the Yamabe sign intrinsically on the conformal infinities of the metrics we constructed.
\begin{lemma}\label{lem:lemma-conformal-infinity-non-locally-conformally-flat}
Each conformal infinity in Theorems
\ref{thm:PE-with-AH-cusp}--\ref{thm:PE-with-Sigma-cusp}
is not locally conformally flat, except in the flat $T^3$ case of Theorem
\ref{thm:PE-with-AH-cusp}, where $g^\sharp$ is flat, and in the product
$\Sigma_{\ttg}\times \bS^1$ case of Theorem
\ref{thm:PE-with-Sigma-cusp}, where $g^\sharp$ is hyperbolic.
\end{lemma}

\begin{proof}
Let
$
g^\flat=\eta^2+g^\natural
$
be the canonical representative of the conformal infinity, where
$g^\natural$ is a metric on $T^2$ or $\Sigma_{\ttg}$. Write
$
d\eta=f\,d\vol_{g^\natural},\quad K=K_{g^\natural},\quad
\nabla=\nabla_{g^\natural}.
$
If $g^\flat$ is locally conformally flat, then its Cotton tensor vanishes,
which gives
$$
\nabla K=3f\nabla f,\qquad
\nabla^2 f=f(f^2-K)g^\natural .
$$
Thus $\nabla f$ is a conformal vector field. On $T^2$, since $\nabla f$ has a
zero, it must vanish identically; on $\Sigma_{\ttg}$, $\ttg\geq2$, there are
no nontrivial conformal vector fields. Hence $f$ is constant. Therefore $K$ is
constant, and either $f=0$ or $K=f^2$. The conclusion follows.
\end{proof}

\begin{lemma}\label{lem:lemma-conformal-infinity-non-positive-yamabe}
Each conformal infinity in Theorems
\ref{thm:PE-with-AH-cusp}--\ref{thm:PE-with-Sigma-cusp}
has negative Yamabe constant, except in the flat $T^3$ case of Theorem
\ref{thm:PE-with-AH-cusp}, where $g^\sharp$ is flat.
\end{lemma}

\begin{proof}
These conformal infinities are aspherical circle bundles, so they admit no
metric of positive scalar curvature by the Schoen--Yau obstruction
\cite{SchoenYau1979Structure}. Hence their Yamabe constants are nonpositive.
If one were zero, a Yamabe minimizer would be scalar-flat; since the manifold
admits no positive scalar curvature, Bourguignon's argument
\cite{Bourguignon1975Stratification} implies that the minimizer is Ricci-flat,
hence flat. Then Lemma
\ref{lem:lemma-conformal-infinity-non-locally-conformally-flat} leaves only
the flat $T^3$ case or the product $\Sigma_{\ttg}\times\bS^1$ case. The latter
is not a flat three-manifold, so only the flat $T^3$ case can occur.
\end{proof}

    \appendix

	\section{Liouville theorems}

	\begin{lemma}\label{lem:Liouville_non_compact_easy}
		Let $u$ satisfy $(e^u)_{\zeta\zeta}+u_{xx}+u_{yy}=0$ on $(0,\infty)\times\mathbb R^2$. If $u-\log(\zeta)\in L^\infty((0,\infty)\times\mathbb R^2)$, then $u(\zeta,x,y)=\log(c\zeta)$ for some $c>0$.
	\end{lemma}

	\begin{proof}
		Set $w:=u-\log(\zeta)$ and $V(Z,z):=e^{w(\frac14|Z|^2,z)}$, where $z=(x,y)\in \mathbb{R}^2, \ Z\in\mathbb{R}^4\backslash \{0\}$. Then $\lambda\leq V\leq \Lambda$ for some $0<\lambda\leq \Lambda<\infty$, and a direct computation shows that $V$ solves $\Delta_ZV+\Delta_z\log V=0$. Equivalently $\mathcal{L} V=0$ with $\mathcal{L}=\operatorname{div}_Z(\nabla_Z)+\operatorname{div}_z(V^{-1}\nabla_z)$, so $\mathcal{L}$ is uniformly elliptic. Since $\{0\}\times\mathbb R^2$ has codimension $4$, a removable singularity theorem (Theorem 2 in \cite{serrin1965removable}) implies that $V$ extends to a bounded global weak solution of $\mathcal{L} V =0$ on $\mathbb R^6$. Hence the Harnack inequality (Theorem 8.20 in \cite{gilbarg_trudinger_2001}) yields that $V$ is constant.
	\end{proof}

	\begin{lemma}\label{lem:Liouville_non_compact}
		Let $u$ satisfy $(e^u)_{\zeta\zeta}+u_{xx}+u_{yy}=0$ on $[0,\infty)\times\mathbb R^2$, with $u|_{\{\zeta=0\}}=0$. If $u-\log(1+\zeta)\in L^\infty([0,\infty)\times\mathbb R^2)$, then $u(\zeta,x,y)=\log(1+c\zeta)$ for some $c>0$.
	\end{lemma}

	\begin{proof}
		Set $w:=u-\log(1+\zeta)$ and $V(Z,z):=e^{w(\frac14|Z|^2-1,z)}$, where $z=(x,y)\in \mathbb{R}^2, \ Z\in\mathbb{R}^4\backslash B_2$. Then $\lambda\leq V\leq \Lambda$ for some $0<\lambda\leq \Lambda<\infty$, $V=1$ on $\{|Z|=2\}$, and a direct computation shows that $V$ solves $\Delta_ZV+\Delta_z\log V=0$. Equivalently $\mathcal{L} V=0$ with $\mathcal{L}=\operatorname{div}_Z(\nabla_Z)+\operatorname{div}_z(V^{-1}\nabla_z)$, so $\mathcal{L}$ is uniformly elliptic.

		We first claim that $V(Z,z)\to c$ as $|Z|\to\infty$, uniformly in $z$, for some constant $c>0$. Otherwise there exist $\varepsilon_0>0$ and points $(Z_j,z_j)$, $(Z_j',z_j')$ with $|Z_j|,|Z_j'|\to\infty$ and $|V(Z_j,z_j)-V(Z_j',z_j')|\ge\varepsilon_0$. Let $L_j:=|Z_j|+|Z_j'|+|z_j-z_j'|$, $\epsilon_j:=2/L_j$, and define $V_j(Z,z):=V(L_jZ,z_j+L_jz)$. Then $V_j$ solves the same equation on $\{|Z|>\epsilon_j\}\times\mathbb R^2$, still satisfies $\lambda\le V_j\le\Lambda$, and equals $1$ on $\{|Z|=\epsilon_j\}$. By local elliptic estimates, after passing to a subsequence, $V_j\to V_\infty$ uniformly on any compact set $D\Subset (\mathbb R^4\setminus\{0\})\times\mathbb R^2$, where $V_\infty$ is bounded and solves $\mathcal{L}_\infty V_\infty=0$, and $\mathcal{L}_\infty=\operatorname{div}_Z(\nabla_Z)+\operatorname{div}_z(V_\infty^{-1}\nabla_z)$. As in the proof of Lemma \ref{lem:Liouville_non_compact_easy}, $V_\infty\equiv c$.

		Now let $Q_j$ be any bounded sequence such that $|Z(Q_j)|/\epsilon_j\to\infty$. We show that $V_j(Q_j)\rightarrow c$. Fix $\delta\gg 1$ and $K>2\delta$, and set
		$\Omega_{j,\delta,K}:=\{\epsilon_j<|Z|<\delta,\ |z|<K\}$.
		Let
		\[
		\psi_j(Z):=\frac{|Z|^{-2}-\delta^{-2}}{\epsilon_j^{-2}-\delta^{-2}},
		\qquad
		\eta_j(\delta,K):=\sup_{\{|Z|=\delta,\ |z|\le K\}}|V_j-c|,
		\]
		and choose $N\ge \sup_j\|V_j-c\|_{L^\infty}$. Then $\psi_j$ is harmonic in the $Z$-variable, satisfies $\psi_j=1$ on $|Z|=\epsilon_j$ and $\psi_j=0$ on $|Z|=\delta$, and
		$\psi_j(Q_j)\sim (\epsilon_j/|Z(Q_j)|)^2\to0$.
		Write $\mathcal{L}_j=\operatorname{div}_Z(\nabla_Z)+\operatorname{div}_z(V_j^{-1}\nabla_z)$ and define
		\[
		W_j:=N\psi_j+\eta_j(\delta,K),\qquad
		v_j^\pm:=(\pm(V_j-c)-W_j)^+.
		\]
		Since $W_j\ge |V_j-c|$ on $|Z|=\epsilon_j$ and on $|Z|=\delta$, the functions $v_j^\pm$ vanish on the $Z$-boundary of $\Omega_{j,\delta,K}$; moreover $0\le v_j^\pm\le 2N$. Because $\mathcal{L}_j(V_j-c)=0$ and $\mathcal{L}_j W_j=0$, we conclude $\mathcal{L}_j v_j^\pm\ge0$ weakly in $\Omega_{j,\delta,K}$. Hence Lemma \ref{lem:annular_cylinder_decay_lemma} yields
		\[
		v_j^\pm(Q_j)\le Ce^{-\gamma K/\delta}.
		\]
		Therefore
		\[
		|V_j(Q_j)-c|
		\le
		v_j^+(Q_j)+v_j^-(Q_j)+W_j(Q_j)
		\le
		2Ce^{-\gamma K/\delta}+N\psi_j(Q_j)+\eta_j(\delta,K).
		\]
		For fixed $\delta$ and $K$, local uniform convergence away from $\{Z=0\}$ gives $\eta_j(\delta,K)\to0$, while $|Z(Q_j)|/\epsilon_j\to\infty$ implies $\psi_j(Q_j)\to0$.
		Taking the limit as $j \to \infty$ and then $K \to \infty$, we conclude that $V_j(Q_j) \to c$.

		Taking $Q_j=P_j:=(Z_j/L_j,0)$ and $Q_j=P_j':=(Z_j'/L_j,(z_j'-z_j)/L_j)$, both of which stay in a fixed bounded set and satisfy $|Z(P_j)|/\epsilon_j=|Z_j|/2\to\infty$ and $|Z(P_j')|/\epsilon_j=|Z_j'|/2\to\infty$, we obtain $V(Z_j,z_j)=V_j(P_j)\to c$ and $V(Z_j',z_j')=V_j(P_j')\to c$, contradicting $|V(Z_j,z_j)-V(Z_j',z_j')|\ge\varepsilon_0$. This proves the claim.

		Finally, for $\varepsilon>0$ set $\phi_\pm(Z):=c\pm\varepsilon+4(1-c\mp\varepsilon)|Z|^{-2}$. Then $\phi_\pm$ solve the same equation, satisfy $\phi_\pm=1$ on $\{|Z|=2\}$, and converge to $c\pm\varepsilon$ as $|Z|\to\infty$. Since $V\to c$ uniformly in $z$, we have $\phi_-\le V\le\phi_+$ for $|Z|$ large, hence everywhere by the maximum principle on truncated domains $\{2\le |Z|\le R\}\times\mathbb R^2$ (e.g., Theorem 1.2 in \cite{busca1999existence}), and letting $\varepsilon\downarrow0$ gives $V(Z,z)=c+4(1-c)|Z|^{-2}$. Since $|Z|^2=4(1+\zeta)$, it follows that
		$
		e^u=(1+\zeta)V=1+c\zeta.
		$
	\end{proof}

	\begin{lemma}\label{lem:annular_cylinder_decay_lemma}
		Let
		\[
		A_{\varepsilon,\delta}:=\{Z\in\mathbb R^4:\varepsilon<|Z|<\delta\},
		\qquad
		Q_K:=A_{\varepsilon,\delta}\times B_K,
		\]
		where $0<\varepsilon<\delta$ and $K\ge 2\delta$. Let
		\[
		Lv:=\operatorname{div}(\mathcal A\nabla v),
		\qquad
		\mathcal A=
		\begin{pmatrix}
			I_4&0\\
			0&aI_2
		\end{pmatrix},
		\qquad
		0<\lambda\le a\le\Lambda.
		\]
		Assume $v\in H^1(Q_K)\cap L^\infty(Q_K)$ satisfies
		\[
		0\le v\le 1,\qquad Lv\ge0\ \text{in }Q_K,
		\qquad
		v=0\ \text{on }\partial A_{\varepsilon,\delta}\times B_K.
		\]
		Then there exist $C,\gamma>0$, depending only on $\lambda,\Lambda$, such that
		\[
		\sup_{A_{\varepsilon,\delta}\times B_{\delta/2}} v
		\le
		C e^{-\gamma K/\delta}.
		\]
	\end{lemma}

	\begin{proof}
		Set
		\[
		F(t):=\int_{A_{\varepsilon,\delta}\times B_t} v^2,
		\qquad 0<t<K.
		\]
		Fix $0<t<t+\delta<K$, and choose $\chi\in C_c^\infty(B_{t+\delta})$ such that
		$\chi\equiv1$ on $B_t$ and $|\nabla_z\chi|\le C\delta^{-1}$.
		Testing the inequality $Lv\ge0$ with $\chi^2v$ gives
		\[
		\int_{A_{\varepsilon,\delta}\times B_t} |\nabla_Z v|^2
		\le
		C\delta^{-2}\bigl(F(t+\delta)-F(t)\bigr).
		\]
		Since $v(\cdot,z)=0$ on $\partial A_{\varepsilon,\delta}$, the Dirichlet
		Poincar\'e inequality on $A_{\varepsilon,\delta}$ yields
		\[
		\int_{A_{\varepsilon,\delta}} v(Z,z)^2\,dZ
		\le
		C\delta^2\int_{A_{\varepsilon,\delta}} |\nabla_Z v(Z,z)|^2\,dZ,
		\]
		with $C$ independent of $\varepsilon$, because
		$\lambda_1(A_{\varepsilon,\delta})\ge \lambda_1(B_\delta^4)\simeq \delta^{-2}$.
		Integrating in $z$ we obtain
		\[
		F(t)\le C\bigl(F(t+\delta)-F(t)\bigr),
		\]
		hence
		\[
		F(t)\le \theta F(t+\delta)
		\]
		for some $\theta\in(0,1)$ depending only on $\lambda,\Lambda$.
		Iterating from $t=\delta$ up to $t\simeq K$ gives
		\[
		F(\delta)\le C e^{-\gamma_0 K/\delta} F(K)
		\le C\delta^4K^2 e^{-\gamma_0 K/\delta},
		\]
		since $0\le v\le1$.

		Finally, the standard local boundedness estimate for nonnegative subsolutions
		of uniformly elliptic divergence-form equations gives
		\[
		\sup_{A_{\varepsilon,\delta}\times B_{\delta/2}} v
		\le
		C\delta^{-3}\|v\|_{L^2(A_{\varepsilon,\delta}\times B_\delta)}.
		\]
		The constant is uniform in $\varepsilon$: after scaling by $\delta$ and using
		$v=0$ on $\partial A_{\varepsilon,\delta}\times B_\delta$, one extends by zero
		across the $Z$-boundary and applies the usual interior estimate in dimension $6$.
		Therefore
		\[
		\sup_{A_{\varepsilon,\delta}\times B_{\delta/2}} v
		\le
		C\delta^{-3}F(\delta)^{1/2}
		\le
		C\frac{K}{\delta}e^{-\gamma_0 K/(2\delta)}
		\le
		C e^{-\gamma_0 K/(4\delta)}.
		\]
	\end{proof}

	\Addresses

\end{document}